\documentclass[12pt,twoside]{article}

\usepackage{amsmath,amssymb,amscd,amsthm}
\usepackage{mathrsfs}
\usepackage{latexsym}
\usepackage{indentfirst}

\usepackage{graphicx}
\usepackage{color}

\usepackage{cite}
\usepackage{enumerate}
\usepackage{authblk}

\usepackage[footnotesize]{subfigure}
\usepackage[footnotesize]{caption2}

\usepackage[colorlinks=false,pdfborder={0 0 0}]{hyperref}
\hypersetup{
  colorlinks=true,
  linkcolor=blue,
  citecolor=blue,
  urlcolor=blue
}

\allowdisplaybreaks[4]

\newtheorem{theorem}{Theorem}[section]

\newtheorem{lemma}{Lemma}[section]
\newtheorem{proposition}{Proposition}[section]

\newtheorem{definition}{Definition}[section]

\makeatletter \oddsidemargin 0.0in \evensidemargin \oddsidemargin
\begin{document}
\raggedbottom

\title{Local Hadamard well-posedness of the 2D primitive equations without thermal diffusion under strict stratification}

\author{%
  {\large Fengjie Luo}\thanks{%
    \emph{E-mail}:
    \href{mailto:Luofengjie664@163.com}{Luofengjie664@163.com}.},
  {\large Quan Wang}\thanks{%
    \emph{E-mail}:
    \href{mailto:xihujunzi@scu.edu.cn}{xihujunzi@scu.edu.cn}.}
  and {\large Zhengdong Du}\thanks{%
    \emph{E-mail}:
    \href{mailto:zdu1985@qq.com}{zdu1985@qq.com}
    (corresponding author).}\\
  {\small Department of Mathematics, Sichuan University}\\
  {\small Chengdu, Sichuan 610064, P. R. China}
}

\date{}

% Compact title layout, local to this title block.
\begingroup
\makeatletter
\renewcommand{\@maketitle}{%
  \newpage
  \null
  \vspace*{-1.5cm}%
  {\centering
    \let\footnote\thanks
    {\LARGE \@title \par}%
    \vskip 0.8em%
    {\large
      \lineskip .5em%
      \@author\par}%
  }%
  \par
  \vskip 0.4em%
}
\makeatother
\setcounter{page}{1}
\maketitle
\endgroup
\thispagestyle{empty}

% Abstract, keywords, and classification: wording unchanged.
% Compact spacing between the abstract heading and its text.
\begingroup
\renewenvironment{abstract}{%
  \small
  \quote
  \setlength{\parskip}{0pt}%
  {\centering\bfseries\abstractname\par}%
  \nobreak\vspace{0.5em}\nobreak
  \noindent\ignorespaces
}{%
  \par\endquote
}

\begin{abstract}
\noindent We study the two-dimensional non-rotating primitive equations with anisotropic viscosity and no thermal diffusion.
For doubly periodic perturbations of a hydrostatic equilibrium with an affine temperature profile, we establish local Hadamard well-posedness
in Sobolev spaces under strict stratification.
The initial horizontal velocity and temperature have sufficient regularity and vanishing vertical means, but need not be small.
The estimates close at the same Sobolev order for these variables despite the horizontal derivative loss in the diagnostic vertical velocity.
Stratification-weighted energy estimates and a horizontal compensation argument avoid an additional temperature derivative
while retaining the full vertical dependence and the vertically homogeneous velocity component determined by periodic compatibility.
A one-sided comparison places higher-order norms of the more regular solution only in terms involving lower-order differences.
This structure allows the Bona--Smith method to yield continuous dependence in the original Sobolev topology.
We obtain a unique maximal strong solution in the strictly stratified class, together with continuous dependence of the diagnostic vertical
velocity and pressure in their natural Sobolev spaces and lower semicontinuity of the maximal lifespan.

\par\vspace{0.4cm}

\par\vspace{0.4cm}
\noindent\textbf{AMS (2010) subject classification:} 35Q86, 35B30, 76D03.
\end{abstract}
\endgroup

% Finish the first page before restoring the body layout.
\clearpage
\flushbottom

% Retain the original body spacing.
\baselineskip 15pt
\parskip 10pt

%=====================================================================
% Section 1. Introduction
%=====================================================================
\section{Introduction}

A basic problem in fluid mechanics is to establish
existence, uniqueness, and continuous dependence on initial data.
For large-scale atmospheric and oceanic flows, the small aspect ratio
leads to the primitive equations, in which hydrostatic balance replaces
vertical momentum evolution. Density stratification and anisotropic
transport are central to these models \cite{Vallis2017}.
Lions, Temam, and Wang developed the global weak-solution theory for
the viscous primitive equations in the early 1990s
\cite{LionsTemamWang1992a,LionsTemamWang1992b}. Subsequent work established
anisotropic strong-solution estimates \cite{GuillenMasmoudiRodriguez2001}
and global strong solutions in two dimensions under suitable boundary
conditions \cite{BreschKazhikhovLemoine2004} and in three dimensions
\cite{CaoTiti2007,Kobelkov2007,KukavicaZiane2007}. Hieber, Hussein, and
Kashiwabara \cite{HieberHusseinKashiwabara2016} developed an
\(L^p\)-theory with heat and salinity diffusion.
Alongside this well-posedness theory, the hydrostatic approximation
has been justified at the weak-solution level \cite{AzeradGuillen2001}
and through strong convergence of the horizontal velocity
\cite{FurukawaEtAl2020}. For hydrodynamic models without an active
temperature, Furukawa et al.\ \cite{FurukawaEtAl2025} distinguished
three limits according to the vertical viscosity scale: primitive
equations with full or only horizontal viscosity, and two-dimensional
Navier--Stokes equations.

The inviscid theory illustrates the role of structural assumptions.
Masmoudi and Wong \cite{MasmoudiWong2012} proved local Sobolev
well-posedness for hydrostatic Euler under the local Rayleigh condition.
Han-Kwan and Nguyen \cite{HanKwanNguyen2016} established nonlinear
ill-posedness near certain shear flows. For hydrostatic Euler--Boussinesq,
nonlinear ill-posedness and failure of the hydrostatic limit were proved
near certain stratified shear profiles violating the Miles--Howard
criterion \cite{BianchiniCotiZelatiErtzbischoff2025}. This criterion
relates the strength of stratification to the shear. Ibrahim et al.\
\cite{IbrahimLinQianTiti2026} studied stable singularity formation
under small analytic perturbations of a blowup profile in a
symmetry-reduced system with non-constant temperature. Stable
stratification also permits the use of isopycnal coordinates.
In this formulation, Bianchini and Duch\^ene \cite{BianchiniDuchene2024} established Sobolev
well-posedness and the hydrostatic limit with isopycnal diffusivity
and no momentum viscosity. For the non-hydrostatic Euler equations
without diffusive regularization, Fradin \cite{Fradin2025} proved
local Sobolev well-posedness in isopycnal coordinates with shear and
impermeable upper and lower boundaries.

For viscous primitive equations coupled to temperature, Cao, Li, and
Titi established strong well-posedness with full momentum viscosity
and only vertical thermal diffusion \cite{CaoLiTiti2014Vertical}. They also treated
systems with only horizontal viscosity and either horizontal \cite{CaoLiTiti2017}
or vertical thermal diffusion \cite{CaoLiTiti2020}. These results cover several
configurations of partial thermal diffusion under the regularity and
compatibility conditions of the respective formulations.
For the related two-dimensional non-hydrostatic Boussinesq equations,
global regularity without thermal diffusion is known with full
velocity viscosity \cite{HouLi2005,Chae2006}. Danchin and Paicu
\cite{DanchinPaicu2011} treated the case of only horizontal viscosity.
Kang, Lee, and Nguyen \cite{KangLeeNguyen2024} obtained global
well-posedness in suitable Sobolev spaces on \(\mathbb R^2\) for a
specific partial-viscosity configuration without thermal diffusion.
They also established small-perturbation stability and large-time
results for systems with thermal diffusion in one direction. On
smooth bounded domains with no-slip conditions, Ayd\i n and Jayanti \cite{AydinJayanti2025}
proved persistence of equal-order integer Sobolev regularity, with
higher-order compatibility conditions. The long-time evolution of
scalar gradients was studied by Kiselev,
Park, and Yao \cite{KiselevParkYao2025}, who proved that the
scalar-gradient \(L^2\)-norm can become arbitrarily large along
sequences of times tending to infinity for viscous solutions satisfying
symmetry and sign assumptions.

For the periodic Boussinesq system, Tao et al.\ \cite{TaoEtAl2020}
studied perturbations of an affine, stably stratified equilibrium.
With normalized buoyancy, their perturbation system is
\[
\left\{
\begin{aligned}
 &\partial_t\mathbf u+\mathbf u\cdot\nabla\mathbf u+\nabla p_{\mathrm B}
   =\nu_{\mathrm B}\Delta\mathbf u+\vartheta\mathbf e_2,\\
 &\partial_t\vartheta+\mathbf u\cdot\nabla\vartheta+b u_2=0,\\
 &\nabla\cdot\mathbf u=0,
\end{aligned}
\right.
\]
where \(\mathbf u=(u_1,u_2)\), \(\mathbf e_2=(0,1)^{\mathsf T}\), and
\(\nu_{\mathrm B},b>0\). They proved global existence without smallness
in their data class, linear stability, and nonlinear stability for
small \(L^2\)-perturbations. For the nonlinear problem, they also
established asymptotic
horizontal stratification and explicit algebraic velocity-decay
estimates under the respective additional assumptions. For the
two-dimensional viscous primitive equations, Jiang and Wang
\cite{JiangWang2025} studied stability and bifurcation near a linear
temperature profile in a periodic channel with full thermal diffusion
and prescribed temperatures at the upper and lower boundaries.

Motivated by these well-posedness and stability results, we ask whether
momentum viscosity and strict stratification suffice for local Hadamard
well-posedness of a doubly periodic hydrostatic system with a
non-diffusive active temperature. In the Boussinesq system above, both
velocity components satisfy evolution equations. In the hydrostatic
system considered here, the vertical velocity is determined
diagnostically through incompressibility and periodic compatibility,
with a loss of one horizontal derivative. Establishing the local theory
therefore requires estimates that account for this diagnostic structure
and the transport of temperature while closing at the initial Sobolev
order.

We consider the following two-dimensional non-rotating system. Let
\(\Omega:=\mathbb T_x\times\mathbb T_z\), with periods \(L_x,L_z>0\). The
perturbations are periodic on \(\Omega\), whereas the total fields are
understood on \(\mathbb T_x\times\mathbb R_z\). The total horizontal and
vertical velocities \((V,W)\), temperature \(T\), and kinematic pressure
\(\pi=\frac{P}{\rho_0}\) satisfy
\begin{equation}\label{eq:1.1.1}
\left\{
\begin{aligned}
 &\partial_tV+V\partial_xV+W\partial_zV
   -\mu_x\partial_{xx}V-\mu_z\partial_{zz}V+\partial_x\pi=0,\\
 &\partial_z\pi=-g\bigl(1-\beta(T-T_0)\bigr),\\
 &\partial_tT+V\partial_xT+W\partial_zT=0,\\
 &\partial_xV+\partial_zW=0.
\end{aligned}
\right.
\end{equation}
Here \(g,\beta,\rho_0>0\) denote the gravitational acceleration, thermal
expansion coefficient, and reference density, respectively. The reference
temperature is \(T_0\), and \(\mu_x,\mu_z>0\) are the horizontal and
vertical momentum viscosities. The hydrostatic relation follows from
\(\partial_zP=-g\rho\) and the linearized equation of state
\(\rho=\rho_0(1-\beta(T-T_0))\). We study this model directly, without
taking a joint hydrostatic and vanishing-diffusivity limit.

For \(\gamma>0\), a hydrostatic equilibrium with affine temperature is
\[
 (V_s,W_s,T_s)=(0,0,T_0+\gamma z),\qquad
 \pi_s(z)=-gz+\frac{g\beta\gamma}{2}z^2.
\]
Thus \(\partial_z\pi_s=-g(1-\beta\gamma z)\). Write \(V=v\), \(W=w\),
\(T=T_s+\theta\), and \(\pi=\pi_s+p\), where \((v,w,\theta,p)\) are
real-valued and doubly periodic. The velocities are vertically periodic,
whereas
\[
 T(x,z+L_z,t)-T(x,z,t)=\gamma L_z,\qquad
 \pi(x,z+L_z,t)-\pi(x,z,t)=\pi_s(z+L_z)-\pi_s(z).
\]
Subtracting the background hydrostatic balance from \eqref{eq:1.1.1} gives
\(\partial_zp=g\beta\theta\), and the transport of \(T_s\) contributes
\(\gamma w\) to the temperature equation. The perturbations therefore
satisfy
\begin{equation}\label{eq:1.1.2}
\left\{
\begin{aligned}
 &\partial_tv+v\partial_xv+w\partial_zv
   -\mu_x\partial_{xx}v-\mu_z\partial_{zz}v+\partial_xp=0,\\
 &\partial_zp=g\beta\theta,\\
 &\partial_t\theta+v\partial_x\theta+w\partial_z\theta+\gamma w=0,\\
 &\partial_xv+\partial_zw=0.
\end{aligned}
\right.
\end{equation}
The reference state is stably stratified, since
\(\partial_z\rho_s=-\rho_0\beta\gamma<0\), with squared
Brunt--V\"ais\"al\"a frequency \(g\beta\gamma>0\). For the perturbed
state, strict stratification means \(R:=\gamma+\partial_z\theta>0\). We
write \(H^r=H^r(\Omega)\), and an overbar denotes vertical averaging.

For every integer \(s\ge4\), we establish local Hadamard well-posedness
for \eqref{eq:1.1.2} with strictly stratified \(H^s\times H^s\) initial
data of zero vertical means, without a smallness assumption. The result
gives a unique maximal strong solution in the strictly stratified class,
continuous dependence in the original Sobolev topology on compact
subintervals of its lifespan, continuity of the diagnostic variables in
their natural spaces, and lower semicontinuity of the maximal lifespan.
The precise statement is Theorem~\ref{thm:hadamard} in
Subsection~\ref{subsec:main-result}.
Three main difficulties arise in establishing this result.

First, the top-order temperature term \(R\partial_x^s w\) is not
controlled in \(L^2\) by the \(H^s\)-norms of \(v,\theta\) alone. Strict
stratification supplies the positive weight \(q=R^{-1}\). For smooth
compatible solutions and \(|\alpha|=s\), hydrostatic balance and
incompressibility yield
\[
 \int_\Omega(\partial_x\partial^\alpha p)\partial^\alpha v
 +g\beta\int_\Omega q(R\partial^\alpha w)\partial^\alpha\theta=0.
\]
The principal coupling therefore cancels from the combined energy
estimate, without a separate top-order bound for \(w\) or thermal
smoothing.

Second, periodic pressure compatibility requires \(\bar\theta=0\), and its
preservation determines \(\bar
w=-\frac{1}{\gamma}\partial_x\overline{v\theta}\). This component cannot
be prescribed independently. Directly estimating \(\partial_x^s\bar w\)
would require an extra horizontal derivative of the averaged flux. We
instead estimate its energy pairing by transferring that derivative to a
velocity factor, where viscous dissipation is available. This compensation
retains the full diagnostic velocity and vertical dependence, without
imposing \(\bar w=0\) or an \(H^{s+1}\)-bound for \(\theta\).

Finally, lower-order stability alone does not yield continuous dependence
in \(H^s\times H^s\). For the stratification-dependent projection
\(\mathcal P_Rf=f-\frac{R}{\gamma}\bar f\), separate estimates of the
projection commutator and projection difference would again require an
extra derivative of the unregularized temperature. We avoid this loss by
estimating them jointly. For a more regular comparison solution \(U^N\),
write \(\delta U^N=U^N-U\), \(\Delta R=R^N-R\), and \(\Delta F=F^N-F\),
where \(U=(v,\theta)\), \(F=v\partial_x\theta+w^0R\),
\(w^0=-\partial_z^{-1}\partial_xv\), and \(F^N\) is defined analogously.
Here \(\partial_z^{-1}\) is the periodic primitive with zero vertical
mean. The highest-order terms satisfy
\[
 (\partial^\alpha R)\overline{\Delta F}
 +(\partial^\alpha\Delta R)\overline{F^N}
 =(\partial^\alpha R^N)\overline{\Delta F}
 +(\partial^\alpha\Delta R)\overline F.
\]
The extra coefficient derivative falls on \(R^N\), not \(R\). Since
\(\partial^\alpha\Delta R=\partial_z\partial^\alpha\delta\theta^N\), the
remaining endpoint has a weighted pairing that extends to a continuous
quadratic form on \(L^2\). Lemma~\ref{lem:projected-commutator} gives the
corresponding joint estimate. The one-sided estimate of
Proposition~\ref{prop:one-sided-comparison} then places the higher-order
norms of \(U^N\) only in terms involving the lower-order difference, not
in the coefficient of the highest-order difference energy. This permits
the Bona--Smith method \cite{BonaSmith1975} to yield same-order continuity
without extra regularity of the reference temperature.

The paper is organized as follows. Section~\ref{sec:diagnostic} derives
the diagnostic and projected formulations, records the analytic estimates,
and states the main result. Section~\ref{sec:friedrichs} constructs the
Friedrichs approximations and proves uniform weighted estimates.
Section~\ref{sec:existence} establishes local strong solutions,
lower-order stability, uniqueness, and continuation.
Section~\ref{sec:bona-smith} proves the higher-order and one-sided
comparison estimates and completes the proof of
Theorem~\ref{thm:hadamard}.

%=====================================================================
% End of Section 1
%=====================================================================

%=====================================================================
% Section 2. Preliminaries and main result
%=====================================================================
\section{Preliminaries and main result}
\label{sec:diagnostic}

We first derive the diagnostic and projected formulations and
record the analytic estimates used in the proofs. We then introduce
the strong-solution class and state the main result.

\subsection{Diagnostic formulation and analytic estimates}
\label{subsec:preliminaries}

All functions are real-valued and periodic on
\(\Omega=\mathbb T_x\times\mathbb T_z\), with fundamental cell
\([0,L_x]\times[0,L_z]\). For a periodic function \(f\), write
\[
 \bar f(x):=\frac{1}{L_z}\int_0^{L_z}f(x,z)\,\mathrm dz,
 \qquad
 \langle f\rangle_x(z):=\frac{1}{L_x}\int_0^{L_x}f(x,z)\,\mathrm dx,
\]
and \(\langle f\rangle:=\langle\bar f\rangle_x\).
The same notation is used for functions independent of one variable.
For \(U=(v,\theta)\), the product norm is
\(\|U\|_{H^r\times H^r}:=\|v\|_{H^r}+\|\theta\|_{H^r}\).

We first derive the diagnostic variables.
Vertical integration of incompressibility gives \(\partial_x\bar v=0\),
and spatial integration of the momentum equation yields
\(\partial_t\langle v\rangle=0\). Thus
\(\bar v=\langle v\rangle\) is conserved, and \(\bar v_0=0\)
implies \(\bar v=0\). Periodicity of \(p\) and
\(\partial_zp=g\beta\theta\) require \(\bar\theta=0\).
Averaging the temperature equation gives
\(\gamma\bar w=-\partial_x\overline{v\theta}\).
Together with \(\partial_zw=-\partial_xv\), this determines
\begin{equation}\label{eq:1.2.1}
 w=-\partial_z^{-1}\partial_xv
   -\frac{1}{\gamma}\partial_x\overline{v\theta}.
\end{equation}
Here \(\partial_z^{-1}\) is the periodic primitive with zero vertical
mean, defined in Fourier variables by \eqref{eq:2.9} below.
The averaged momentum equation gives
\(\partial_x\bar p=-\partial_x\overline{v^2}\). With the normalization
\begin{equation}\label{eq:1.2.2}
 \langle p\rangle=0,
\end{equation}
the hydrostatic relation therefore yields
\begin{equation}\label{eq:1.2.3}
 p=g\beta\partial_z^{-1}\theta-\overline{v^2}
   +\langle v^2\rangle.
\end{equation}
Neither \(w\) nor \(p\) requires independent initial data.

On the compatibility class \(\bar v=\bar\theta=0\),
system~\eqref{eq:1.1.2} is therefore equivalent to
\begin{equation}\label{eq:2.1}
\left\{
\begin{aligned}
 &\partial_t v+v\partial_xv+w\partial_zv
   -\mu_x\partial_{xx}v-\mu_z\partial_{zz}v
   +g\beta\partial_x\partial_z^{-1}\theta
   -\partial_x\overline{v^2}=0,\\
 &\partial_t\theta+v\partial_x\theta+w\partial_z\theta+\gamma w=0,
\end{aligned}
\right.
\end{equation}
where \(w\) is given by \eqref{eq:1.2.1}.
Conversely, a solution of \eqref{eq:2.1} in this compatibility class,
together with the diagnostic variables defined by
\eqref{eq:1.2.1} and \eqref{eq:1.2.3}, satisfies
\eqref{eq:1.1.2} and \eqref{eq:1.2.2}.

To obtain the projected temperature equation, set
\begin{equation}\label{eq:2.2}
 R:=\gamma+\partial_z\theta,
 \qquad w^0:=-\partial_z^{-1}\partial_xv,
 \qquad \bar w:=-\frac{1}{\gamma}\partial_x\overline{v\theta},
 \qquad w=w^0+\bar w.
\end{equation}
Then \(\overline{w^0}=0\), \(\partial_zw^0=-\partial_xv\), and
\(\bar R=\gamma\). Hence the operator
\begin{equation}\label{eq:2.6}
 \mathcal P_Rf:=f-\frac{R}{\gamma}\bar f
\end{equation}
is a projection onto the zero-vertical-mean subspace:
\(\overline{\mathcal P_Rf}=0\), and \(\mathcal P_Rf=f\) whenever
\(\bar f=0\). In particular, \(\mathcal P_R^2=\mathcal P_R\).
For \(F:=v\partial_x\theta+w^0R\), integration by parts in \(z\) gives
\begin{equation}\label{eq:2.3}
 \bar F=\partial_x\overline{v\theta}=-\gamma\bar w,
 \qquad \mathcal P_RF=v\partial_x\theta+wR.
\end{equation}
The temperature equation consequently takes the form
\begin{equation}\label{eq:2.7}
 \partial_t\theta+
 \mathcal P_R\bigl(v\partial_x\theta+w^0\partial_z\theta
                  +\gamma w^0\bigr)=0.
\end{equation}
If \(R\in L^\infty(\Omega)\), \(R\ge\kappa>0\), and
\(f,\psi\in L^2(\Omega)\) with \(\bar\psi=0\), then
\begin{equation}\label{eq:2.4}
 \int_\Omega\frac{\mathcal P_Rf}{R}\,\psi\,\mathrm dx\,\mathrm dz
 =\int_\Omega\frac{f}{R}\,\psi\,\mathrm dx\,\mathrm dz.
\end{equation}
The difference vanishes because \(\bar\psi=0\).
In particular, this identity applies to differentiated temperatures.

We now specify the Fourier convention. Let
\(\mathfrak K=\frac{2\pi}{L_x}\mathbb Z\times\frac{2\pi}{L_z}\mathbb Z\)
and \(\mathfrak K_*:=\{(k,m)\in\mathfrak K:m\ne0\}\).
For \(f\in L^2(\Omega)\), set
\[
 \widehat f(k,m):=\frac{1}{L_xL_z}\int_\Omega
       f(x,z)e^{-i(kx+mz)}\,\mathrm dx\,\mathrm dz.
\]
Then \(f=\sum_{(k,m)\in\mathfrak K}\widehat f(k,m)e^{i(kx+mz)}\)
in \(L^2(\Omega)\). For \(r\in\mathbb R\), the Sobolev norm is
\begin{equation}\label{eq:2.8}
 \|f\|_{H^r}^2
 =L_xL_z\sum_{(k,m)\in\mathfrak K}
        (1+k^2+m^2)^r|\widehat f(k,m)|^2.
\end{equation}
The periodic primitive is defined by
\begin{equation}\label{eq:2.9}
 \widehat{\partial_z^{-1}f}(k,m)=
 \begin{cases}
  (im)^{-1}\widehat f(k,m),&m\ne0,\\
  0,&m=0.
 \end{cases}
\end{equation}
It commutes with spatial derivatives and satisfies
\(\partial_z\partial_z^{-1}f=f-\bar f\) and
\(\overline{\partial_z^{-1}f}=0\). In particular, it is the unique
zero-vertical-mean primitive when \(\bar f=0\).
Since \(|m|\ge\frac{2\pi}{L_z}\) on \(\mathfrak K_*\),
for every \(r\ge0\),
\begin{equation}\label{eq:2.10}
 \|\partial_z^{-1}f\|_{H^r}
 +\|\partial_z^{-1}\partial_xf\|_{H^{r-1}}
 \le C\|f\|_{H^r}.
\end{equation}
The second term is understood in \(H^{-1}\) when \(r=0\).
For \(\bar f=0\), the integral representation of the primitive also gives
\begin{equation}\label{eq:2.11}
 \|\partial_z^{-1}f\|_{L^\infty}\le C\|f\|_{L^\infty}.
\end{equation}
Under the same mean constraint, the nonzero vertical frequencies yield
the Poincar\'e estimate
\begin{equation}\label{eq:2.12}
 \|f\|_{H^{r+1}}\le C\bigl(
       \|\partial_xf\|_{H^r}+\|\partial_zf\|_{H^r}\bigr),
 \qquad r\ge0.
\end{equation}
The constants in \eqref{eq:2.10}--\eqref{eq:2.12} depend only on the
regularity index and the periods.

For \(\alpha=(\alpha_1,\alpha_2)\in\mathbb N_0^2\), write
\(\partial^\alpha:=\partial_x^{\alpha_1}\partial_z^{\alpha_2}\) and
\(|\alpha|:=\alpha_1+\alpha_2\). For each integer \(r\ge0\),
\begin{equation}\label{eq:2.13}
 C^{-1}\|f\|_{H^r}^2
 \le\sum_{|\alpha|\le r}\|\partial^\alpha f\|_{L^2}^2
 \le C\|f\|_{H^r}^2.
\end{equation}
The energy estimates use this equivalent derivative norm with unit
weights. The relation \(\nu\le\alpha\) is understood componentwise, and
\(\binom{\alpha}{\nu}:=\binom{\alpha_1}{\nu_1}
\binom{\alpha_2}{\nu_2}\). We use \(H_x^r\), \(L_x^p\), and
\(W_x^{j,\infty}\) for spaces on \(\mathbb T_x\), and identify their
functions with their \(z\)-independent extensions when needed.
On a time interval \([0,T]\) fixed by the context, we write
\(C_tX=C([0,T];X)\), \(L_t^pX=L^p(0,T;X)\), and
\(W_t^{1,p}X=W^{1,p}(0,T;X)\) for a Banach space \(X\).

The nonlinear estimates are the standard Sobolev product, commutator,
and composition bounds \cite{BahouriCheminDanchin2011,KatoPonce1988}.
They are understood whenever the norms on the right-hand sides are
finite. For \(r\ge0\), the tame product estimate is
\begin{equation}\label{eq:2.14}
 \|fg\|_{H^r}\le C\bigl(
       \|f\|_{L^\infty}\|g\|_{H^r}
       +\|g\|_{L^\infty}\|f\|_{H^r}\bigr).
\end{equation}
For an integer \(r=|\alpha|\ge1\), with
\([\partial^\alpha,f]g:=\partial^\alpha(fg)-f\partial^\alpha g\),
\begin{equation}\label{eq:2.15}
 \|[\partial^\alpha,f]g\|_{L^2}\le C\bigl(
       \|\nabla f\|_{L^\infty}\|g\|_{H^{r-1}}
       +\|f\|_{H^r}\|g\|_{L^\infty}\bigr).
\end{equation}
When both endpoint terms are removed from the Leibniz expansion, one has
\begin{equation}\label{eq:2.5}
 \left\|
 \sum_{\substack{0<\nu\le\alpha\\\nu\ne\alpha}}
 \binom{\alpha}{\nu}
 (\partial^\nu f)(\partial^{\alpha-\nu}g)
 \right\|_{L^2}
 \le C\bigl(\|\nabla f\|_{L^\infty}\|g\|_{H^{r-1}}
       +\|\nabla g\|_{L^\infty}\|f\|_{H^{r-1}}\bigr),
\end{equation}
for every integer \(r\ge2\) and \(|\alpha|\le r\).
For \(|\alpha|\le1\), the sum is empty. This estimate controls the
mixed terms using only \(r-1\) derivatives of each factor.
We also use the balanced product estimate: if
\(\sigma_1,\sigma_2\ge0\) and \(\sigma_1+\sigma_2>1\), then
\begin{equation}\label{eq:2.16}
 \|fg\|_{L^2}\le C\|f\|_{H^{\sigma_1}}\|g\|_{H^{\sigma_2}}.
\end{equation}
In particular, for \(r>1\),
\begin{equation}\label{eq:2.17}
 \|fg\|_{H^r}\le C\|f\|_{H^r}\|g\|_{H^r}.
\end{equation}

The Sobolev embeddings depend on the dimension. For every integer
\(j\ge0\),
\begin{equation}\label{eq:2.18}
 \|f\|_{W^{j,\infty}(\Omega)}\le C\|f\|_{H^r(\Omega)},
 \qquad r>j+1,
\end{equation}
whereas, on \(\mathbb T_x\),
\begin{equation}\label{eq:2.19}
 \|a\|_{W_x^{j,\infty}}\le C\|a\|_{H_x^r},
 \qquad r>j+\frac{1}{2}.
\end{equation}
Thus \(H^2(\Omega)\hookrightarrow L^\infty(\Omega)\) and
\(H^s(\Omega)\hookrightarrow W^{2,\infty}(\Omega)\) for every integer
\(s\ge4\). In one dimension, \(H_x^{s-1}\hookrightarrow
W_x^{2,\infty}\) and \(H_x^s\hookrightarrow W_x^{3,\infty}\)
for the same range of \(s\).

To control the reciprocal stratification, let \(r>1\) and
\(a_1,a_2\in H^r(\Omega)\) satisfy \(a_j\ge\kappa>0\) for
\(j=1,2\) and \(\|a_1\|_{H^r}+\|a_2\|_{H^r}\le M\).
The composition estimates and \eqref{eq:2.17} imply
\begin{equation}\label{eq:2.20}
 \|a_j^{-1}\|_{H^r}\le C,
 \qquad
 \|a_1^{-1}-a_2^{-1}\|_{H^r}\le C\|a_1-a_2\|_{H^r}.
\end{equation}
Here \(C\) depends only on \(r,\kappa,M\), and the periods.
The difference estimate also follows from
\(a_1^{-1}-a_2^{-1}=a_1^{-1}(a_2-a_1)a_2^{-1}\).

Finally, vertical averaging is bounded on Sobolev spaces:
\begin{equation}\label{eq:2.21}
 \|\overline F\|_{H_x^r}\le C\|F\|_{H^r},
 \qquad r\ge0.
\end{equation}
Since averaging commutes with horizontal derivatives, for every
integer \(r\ge0\),
\begin{equation}\label{eq:2.22}
 \|\partial_x^r\overline{FG}\|_{L_x^2}\le C\|FG\|_{H^r}.
\end{equation}
These inequalities are combined with \eqref{eq:2.14} and
\eqref{eq:2.19} to estimate the averaged fluxes.
Throughout the paper, \(C\) denotes a positive constant that may change
from line to line. Dependence on the relevant parameters and uniformity
with respect to approximation parameters are stated when needed.

\subsection{Main result}
\label{subsec:main-result}

With the notation of Subsection~\ref{subsec:preliminaries}, we now
specify the initial-data and strong-solution classes.
For every integer \(s\ge4\), define
\[
\begin{aligned}
 \mathcal Z^s:={}&
 \bigl\{(v,\theta)\in H^s\times H^s:
          \bar v=\bar\theta=0\bigr\},\\
 \mathcal X_\kappa^s:={}&
 \bigl\{(v,\theta)\in\mathcal Z^s:
          \inf_\Omega(\gamma+\partial_z\theta)\ge\kappa\bigr\},
 \qquad \kappa>0,\\
 \mathcal X^s:={}&\bigcup_{\kappa>0}\mathcal X_\kappa^s.
\end{aligned}
\]
The stratification condition is well defined pointwise by Sobolev
embedding. We endow \(\mathcal X^s\) with the relative
\(H^s\times H^s\)-topology and use the following solution class.

\begin{definition}[Strong solution]\label{def:strong-solution}
Let \(s\ge4\) be an integer and let \(T>0\).
A periodic quadruple \((v,w,\theta,p)\) is an \(H^s\)-strong solution of
\eqref{eq:1.1.2} on \([0,T]\) if
\begin{align*}
 &v,\theta\in C([0,T];H^s),
 \qquad v\in L^2(0,T;H^{s+1}),
 \\
 &\partial_t v\in L^2(0,T;H^{s-1}),
 \qquad \partial_t\theta\in C([0,T];H^{s-1}),
\end{align*}
with \(\bar v=\bar\theta=0\) and
\(\gamma+\partial_z\theta\ge\kappa\) on \([0,T]\times\Omega\) for some
\(\kappa>0\). The variables \(w\) and \(p\) are given by
\eqref{eq:1.2.1} and \eqref{eq:1.2.3}, and the horizontal momentum and
temperature equations in \eqref{eq:1.1.2} hold almost everywhere on
\((0,T)\times\Omega\). In particular,
\[
 w\in C([0,T];H^{s-1}),
 \qquad p\in C([0,T];H^s).
\]
\end{definition}

The main result is stated next.

\begin{theorem}[Local Hadamard well-posedness]
\label{thm:hadamard}
Let \(s\ge4\) be an integer and let
\(U_0=(v_0,\theta_0)\in\mathcal X^s\).
Then the following assertions hold.
\begin{enumerate}
\item[\textup{(i)}]
System~\eqref{eq:1.1.2} admits a unique maximal solution
\((v,w,\theta,p)\) on \([0,T_{\max}(U_0))\), where
\(T_{\max}(U_0)\in(0,\infty]\).
For every \(0<T<T_{\max}(U_0)\), its restriction to \([0,T]\) is an
\(H^s\)-strong solution with initial datum \((v,\theta)|_{t=0}=U_0\).
Maximality is understood within the strictly stratified strong-solution
class of Definition~\ref{def:strong-solution}.

\item[\textup{(ii)}]
The compatibility conditions \(\bar v=\bar\theta=0\) are preserved.
The diagnostic variables \(w\) and \(p\) are uniquely determined by
\eqref{eq:1.2.1} and \eqref{eq:1.2.3}, with the normalization
\eqref{eq:1.2.2}.

\item[\textup{(iii)}]
For every \(0<T<T_{\max}(U_0)\), there exists \(\kappa_T>0\) such that,
on \([0,T]\times\Omega\),
\begin{equation*}
 \gamma+\partial_z\theta\ge\kappa_T.
\end{equation*}

\item[\textup{(iv)}]
Let \(U_{0,n}\in\mathcal Z^s\) satisfy
\begin{equation*}
 U_{0,n}\longrightarrow U_0
 \qquad\text{in }H^s\times H^s.
\end{equation*}
For every \(0<T<T_{\max}(U_0)\) and all sufficiently large \(n\),
\begin{equation*}
 U_{0,n}\in\mathcal X^s,
 \qquad T<T_{\max}(U_{0,n}).
\end{equation*}
Writing \(U_n(t)=U(t;U_{0,n})\), \(U(t)=U(t;U_0)\), and denoting the
corresponding diagnostic variables by \((w_n,p_n)\) and \((w,p)\), one has
\begin{alignat*}{2}
 U_n \longrightarrow{}& U
 &\qquad& \text{in } C([0,T];H^s\times H^s),\\
 w_n \longrightarrow{}& w
 &\qquad& \text{in } C([0,T];H^{s-1}),\\
 p_n \longrightarrow{}& p
 &\qquad& \text{in } C([0,T];H^s).
\end{alignat*}
All pressures are normalized by \eqref{eq:1.2.2}.

\item[\textup{(v)}]
The maximal lifespan is lower semicontinuous. More precisely, for every
sequence as in \textup{(iv)}, there exists \(n_0\ge1\) such that
\(U_{0,n}\in\mathcal X^s\) for all \(n\ge n_0\), and
\begin{equation*}
 T_{\max}(U_0)\le\liminf_{n\to\infty}T_{\max}(U_{0,n}).
\end{equation*}
\end{enumerate}
\end{theorem}

%=====================================================================
% End of Section 2
%=====================================================================

%=====================================================================
% Section 3. Friedrichs approximation and uniform estimates
%=====================================================================

\section{Friedrichs approximation and uniform estimates}
\label{sec:friedrichs}

Fix an integer \(s\ge4\) and \(U_0\in\mathcal X_{\kappa_0}^s\).
We construct Fourier--Friedrichs approximations preserving the periodic
compatibility conditions. A weighted energy estimate then yields a
common existence interval without requiring an additional temperature
derivative.

Choose a real even function \(\chi\in C_c^\infty(\mathbb R^2)\) such that
\(0\le\chi\le1\), \(\chi=1\) on the unit ball, and
\(\operatorname{supp}\chi\) is contained in the ball of radius two.
For \(K\ge1\), define
\begin{equation}\label{eq:3.1.1}
 J_Kf:=\sum_{(k,m)\in\mathfrak K}
 \chi\!\left(\frac{k}{K},\frac{m}{K}\right)
 \widehat f(k,m)e^{i(kx+mz)}.
\end{equation}
The operator \(J_K\) is self-adjoint and has finite rank. It preserves
real-valuedness and commutes with derivatives, vertical averaging, and
\(\partial_z^{-1}\). It is uniformly bounded on \(H^r\),
\(r\in\mathbb R\), and on \(L^\infty\), and converges strongly to the
identity on each \(H^r\). These are the standard Fourier--Friedrichs
properties \cite{Temam2001}. Neither idempotence nor preservation of
positivity is assumed.

Set \(U_{0,K}:=J_KU_0\). By the Fourier norm \eqref{eq:2.8}, the support
of \(1-\chi(\frac{\cdot}{K})\), and \eqref{eq:2.18},
\[
 \|\partial_z(J_K-I)\theta_0\|_\infty
 \le C\|(J_K-I)\theta_0\|_{H^3}
 \le CK^{-1}\|\theta_0\|_{H^4}.
\]
Thus there exists \(K_0\ge1\), depending only on the initial
\(H^s\)-bound, \(\kappa_0^{-1}\), the periods, and \(\chi\), such that
\begin{equation}\label{eq:3.1.2}
 R_{0,K}:=\gamma+\partial_zJ_K\theta_0\ge\frac{\kappa_0}{2},
 \qquad K\ge K_0.
\end{equation}
Only these cutoffs are considered below.
For \(U_K=(v_K,\theta_K)\), set \(R_K=\gamma+\partial_z\theta_K\)
and reconstruct \(w_K^0,\bar w_K,w_K,p_K\) by
\eqref{eq:2.2} and \eqref{eq:1.2.3}. The approximate equations are
\begin{equation}\label{eq:3.1.3}
\left\{
\begin{aligned}
 &\partial_tv_K+J_K(v_K\partial_xv_K+w_K\partial_zv_K+\partial_xp_K)
       -\mu_x\partial_{xx}v_K-\mu_z\partial_{zz}v_K=0,\\
 &\partial_t\theta_K+
       J_K\mathcal P_{R_K}(v_K\partial_x\theta_K+w_K^0R_K)=0,\\
 &(v_K,\theta_K)|_{t=0}=U_{0,K}.
\end{aligned}
\right.
\end{equation}
This system defines a smooth finite-dimensional vector field on the
zero-vertical-mean subspace of \(\operatorname{Ran}J_K\), taken for both
unknowns. Indeed, on this subspace, incompressibility and the diagnostic
pressure formula imply
\[
 \overline{v_K\partial_xv_K+w_K\partial_zv_K+\partial_xp_K}
 =\partial_x\bigl(\overline{v_K^2}+\bar p_K\bigr)=0.
\]
Also, \(\overline{\mathcal P_{R_K}f}=0\) by \eqref{eq:2.6}.
Hence both components of the vector field have zero vertical mean,
and the viscous terms leave \(\operatorname{Ran}J_K\) invariant.
Finite-dimensional ODE theory therefore yields a unique maximal smooth
solution on \([0,T_K^{\max})\). The diagnostic reconstruction preserves
incompressibility, hydrostatic balance, and the pressure normalization.

We next establish estimates independent of \(K\). In the energy identities,
integer-order Sobolev norms are taken in the equivalent derivative form
of \eqref{eq:2.13}, with unit weights. Define
\begin{equation}\label{eq:3.2.1}
\begin{aligned}
 E_{s,K}:={}&\|v_K\|_{H^s}^2+
       g\beta\sum_{|\alpha|\le s}\int_\Omega
       R_K^{-1}|\partial^\alpha\theta_K|^2,\\
 D_{s,K}:={}&\mu_x\|\partial_xv_K\|_{H^s}^2
                  +\mu_z\|\partial_zv_K\|_{H^s}^2.
\end{aligned}
\end{equation}
Integrals over \(\Omega\) are with respect to \(\mathrm dx\,\mathrm dz\)
unless otherwise indicated. On any interval where \(R_K\ge\kappa>0\)
and \(\|v_K\|_{W^{2,\infty}}+\|\theta_K\|_{W^{2,\infty}}\le M\),
we have \(\kappa\le R_K\le\gamma+M\). Consequently,
\eqref{eq:2.12}--\eqref{eq:2.13} and \(\bar v_K=0\) yield
\begin{equation}\label{eq:3.2.2}
 C^{-1}\|U_K\|_{H^s\times H^s}^2\le E_{s,K}
 \le C\|U_K\|_{H^s\times H^s}^2,\qquad
 D_{s,K}\ge c\|v_K\|_{H^{s+1}}^2.
\end{equation}
Unless otherwise stated, all constants below are independent of \(K\).

The diagnostic estimates follow directly from the tools in
Section~\ref{sec:diagnostic}. If \(\bar v=\bar\theta=0\) and
\(\|v\|_{W^{2,\infty}}+\|\theta\|_{W^{2,\infty}}\le M\), then, for
every integer \(r\ge0\) for which the right-hand sides are finite,
\begin{equation}\label{eq:3.2.3}
 \|w\|_{W^{1,\infty}}\le C,\quad
 \|w^0\|_{H^r}\le C\|v\|_{H^{r+1}},\quad
 \|\bar w\|_{H_x^r}\le C(\|v\|_{H^{r+1}}+\|\theta\|_{H^{r+1}}).
\end{equation}
Indeed, \eqref{eq:2.11}, together with
\(\partial_xw^0=-\partial_z^{-1}\partial_{xx}v\) and
\(\partial_zw^0=-\partial_xv\), bounds \(w^0\) in \(W^{1,\infty}\).
The homogeneous component satisfies
\[
\|\bar w\|_{W_x^{1,\infty}}\le C\|v\theta\|_{W^{2,\infty}}\le CM^2.
\]
The second estimate in \eqref{eq:3.2.3} is \eqref{eq:2.10} at order
\(r+1\), and the last follows from
\[
 \|\bar w\|_{H_x^r}\le C\|\overline{v\theta}\|_{H_x^{r+1}}
 \le C\|v\theta\|_{H^{r+1}}
 \le C(\|v\|_{H^{r+1}}+\|\theta\|_{H^{r+1}}),
\]
using \eqref{eq:2.21} and \eqref{eq:2.14}.

The weighted energy argument requires a cutoff commutator acting on
\(H^{-1}\), rather than only on \(L^2\). We record this estimate together
with the identity for the regularized transport terms.

\begin{lemma}\label{lem:friedrichs-commutator}
Let \(a\in W^{1,\infty}(\Omega)\), \(f\in H^{-1}(\Omega)\),
\(h\in L^2(\Omega)\), and \(j\in\{x,z\}\). Then
\begin{equation}\label{eq:3.2.4}
 \|[J_K,a]f\|_{L^2}\le C\|\nabla a\|_{L^\infty}\|f\|_{H^{-1}},\qquad
 \|a\partial_jh\|_{H^{-1}}\le C\|a\|_{W^{1,\infty}}\|h\|_{L^2},
\end{equation}
with \(C\) independent of \(K\ge1\). Here
\([J_K,a]f:=J_K(af)-aJ_Kf\), and multiplication in \(H^{-1}\) is
understood by duality. If \(a,h\) are real-valued and \(h\in H^1(\Omega)\),
then
\begin{equation}\label{eq:3.2.5}
 2\int_\Omega J_K(a\partial_jh)h
 =\int_\Omega[J_K,a]\partial_jh\,h
       -\int_\Omega J_K((\partial_ja)h)h.
\end{equation}
\end{lemma}
\begin{proof}
Let \(\mathcal K_K\) be the periodized Schwartz kernel of \(J_K\), and
let \(d\) denote the periodic distance. Scaling before periodization gives
\[
 \sup_{K\ge1}\left(\|\mathcal K_K\|_{L^1}
       +\int_\Omega d(y,0)|\nabla\mathcal K_K(y)|\,\mathrm dy\right)<\infty.
\]
For smooth \(g\), integration by parts in the kernel formula yields
\[
\begin{aligned}
 {}[J_K,a]\partial_jg(\xi)={}&
 \int_\Omega\partial_j\mathcal K_K(y)
       \bigl(a(\xi-y)-a(\xi)\bigr)g(\xi-y)\,\mathrm dy\\
 &-\int_\Omega\mathcal K_K(y)(\partial_ja)(\xi-y)g(\xi-y)\,\mathrm dy,
\end{aligned}
\]
where \(\xi=(x,z)\). The kernel bounds and
\(|a(\xi-y)-a(\xi)|\le d(y,0)\|\nabla a\|_\infty\), followed by
Young's inequality, imply
\[
 \|[J_K,a]g\|_2+\sum_{j\in\{x,z\}}\|[J_K,a]\partial_jg\|_2
 \le C\|\nabla a\|_\infty\|g\|_2.
\]
This estimate extends to \(g\in L^2\) by density. For \(f\in H^{-1}\),
put \(\phi=(I-\Delta)^{-1}f\). Then
\[
 f=\phi+\sum_{j\in\{x,z\}}\partial_j(-\partial_j\phi),
 \qquad \|\phi\|_2+\|\nabla\phi\|_2\le C\|f\|_{H^{-1}}.
\]
Applying the preceding estimate to \(\phi\) and \(-\partial_j\phi\)
proves the first inequality in \eqref{eq:3.2.4}. The second follows from
\(\langle a\partial_jh,\psi\rangle=-\int_\Omega h\partial_j(a\psi)\)
for \(\psi\in H^1\). Finally, self-adjointness of \(J_K\) and integration
by parts in \(\int_\Omega a(\partial_jh)J_Kh\) prove \eqref{eq:3.2.5}.
\end{proof}

The vertically homogeneous velocity produces a separate horizontal
endpoint. We give its compensated form here, since the same expression
will be used in the difference estimates. For smooth \(a=a(x)\), \(u\),
and an integer \(m\ge3\), set
\[
 \mathcal H_\alpha[a,u]:=[\partial^\alpha,a]\partial_zu
       -\mathbf1_{\{\alpha=(m,0)\}}(\partial_x^ma)\partial_zu,
 \qquad |\alpha|\le m.
\]
The endpoint-free estimate \eqref{eq:2.5}, applied to \(a,\partial_zu\),
controls the terms in which both factors are differentiated. Any
remaining endpoint has at most \(m-1\) derivatives on \(a\). Thus
\begin{equation}\label{eq:3.2.6}
 \sum_{|\alpha|\le m}\|\mathcal H_\alpha[a,u]\|_2
 \le C\bigl(\|\partial_xa\|_{L_x^\infty}\|u\|_{H^m}
           +\|a\|_{H_x^{m-1}}\|u\|_{W^{2,\infty}}\bigr),
\end{equation}
where \(a\) is identified with its \(z\)-independent extension. The exact
decomposition is
\begin{equation}\label{eq:3.2.7}
 \partial^\alpha(a\partial_zu)=a\partial_z\partial^\alpha u
 +\mathbf1_{\{\alpha=(m,0)\}}(\partial_x^ma)\partial_zu
 +\mathcal H_\alpha[a,u].
\end{equation}
In particular, the only excluded derivative of order \(m\) on the
coefficient occurs when \(\alpha=(m,0)\).

For \(A=A(x)\), define
\(\mathcal B_m(A;u,h):=\frac{1}{\gamma}\int_{\mathbb T_x}
\partial_x^mA\,\partial_x\bigl(\int_0^{L_z}
\partial_zu\,\partial_x^mh\,\mathrm dz\bigr)\,\mathrm dx\).
With \(a=-\frac{1}{\gamma}\partial_xA\), periodic integration by
parts in \(x\) yields
\begin{equation}\label{eq:3.2.8}
 \int_\Omega(\partial_x^ma)(\partial_zu)\partial_x^mh
 =\mathcal B_m(A;u,h).
\end{equation}
By the product rule and the Cauchy--Schwarz inequality,
\begin{equation}\label{eq:3.2.9}
 |\mathcal B_m(A;u,h)|\le C\|A\|_{H_x^m}
 \bigl(\|\partial_{xz}u\|_\infty\|h\|_{H^m}
       +\|\partial_zu\|_\infty\|h\|_{H^{m+1}}\bigr).
\end{equation}
This formula defines the endpoint whenever the right-hand side is
finite. In time-dependent applications, convergence of the first factor
in \eqref{eq:3.2.8} in \(L_t^\infty L_x^2\) and of the differentiated
inner integral in \(L_t^2L_x^2\) implies convergence of the pairings in
\(L^1(0,T)\). In particular, this holds if
\(A_\ell\to A\) in \(C_tH_x^m\), \(u_\ell\to u\) in
\(C_tW^{2,\infty}\), and \(h_\ell\to h\) in
\(C_tH^m\cap L_t^2H^{m+1}\). Thus the compensated form also passes to
strong limits at the regularity used below.

We can now combine the momentum estimate with the weighted temperature
identity. The cutoff commutator is estimated before passing to the
limit, so no projection property of \(J_K\) is needed.

\begin{lemma}\label{lem:weighted-energy}
Let \([0,T]\subset[0,T_K^{\max})\). Suppose that
\(R_K\ge\kappa>0\) and
\(\|v_K\|_{W^{2,\infty}}+\|\theta_K\|_{W^{2,\infty}}\le M\)
on this interval. Then, for every \(0\le t\le T\),
\begin{equation}\label{eq:3.2.10}
 E_{s,K}(t)+c\int_0^t\|v_K\|_{H^{s+1}}^2\,\mathrm d\rho
 \le E_{s,K}(0)+C\int_0^t(1+E_{s,K})^2\,\mathrm d\rho,
\end{equation}
where \(c,C>0\) are independent of \(K\).
\end{lemma}
\begin{proof}
We suppress the subscript \(K\) on the unknowns and write \(q=R^{-1}\).
The finite-dimensional solution and its diagnostic variables are smooth,
so all the following identities are justified before any estimates are
applied.

We first test the differentiated momentum equation by
\(\partial^\alpha v\), for \(|\alpha|\le s\). The hydrostatic pressure
satisfies
\begin{equation}\label{eq:3.2.11}
 g\beta\int_\Omega J_K\partial_x\partial_z^{-1}\partial^\alpha\theta\,
              \partial^\alpha v
 =-g\beta\int_\Omega J_K\partial^\alpha\theta\,\partial^\alpha w.
\end{equation}
Indeed, \(\partial_x\partial_z^{-1}\) is self-adjoint, commutes with
\(J_K\), and maps \(v\) to \(-w^0\). The homogeneous components of
\(p\) and \(w\) have zero pairing with the corresponding differentiated
velocity and temperature, since both have zero vertical mean.
The principal transports are bounded by \(CE_{s,K}\) using
\eqref{eq:3.2.5} and \eqref{eq:3.2.3}. For the commutators,
\eqref{eq:2.15}, \eqref{eq:3.2.3}, and \eqref{eq:3.2.6} imply
\[
 \sum_{|\alpha|\le s}\bigl(\|[\partial^\alpha,v]\partial_xv\|_2
       +\|[\partial^\alpha,w^0]\partial_zv\|_2
       +\|\mathcal H_\alpha[\bar w,v]\|_2\bigr)
 \le C\bigl(E_{s,K}^{\frac{1}{2}}+\|v\|_{H^{s+1}}\bigr).
\]
Here \(\|w^0\|_{H^s}\le C\|v\|_{H^{s+1}}\), whereas
\(\|\bar w\|_{H_x^{s-1}}\) requires only the \(H^s\)-norms of
\(v,\theta\). Pairing with \(J_K\partial^\alpha v\), using
\eqref{eq:3.2.2}, and applying Young's inequality bounds these terms by
\(\eta D_{s,K}+C_\eta E_{s,K}\).

It remains to estimate the horizontal endpoint. Put
\(A=\overline{v\theta}\). By \eqref{eq:2.14} and \eqref{eq:2.21},
\(\|A\|_{H_x^s}\le CE_{s,K}^{\frac{1}{2}}\). In
\eqref{eq:3.2.8}, take \(m=s\), \(u=v\), and \(h=J_Kv\).
Then \eqref{eq:3.2.9} yields
\[
 |\mathcal B_s(A;v,J_Kv)|
 \le CE_{s,K}^{\frac{1}{2}}
       \bigl(E_{s,K}^{\frac{1}{2}}+\|v\|_{H^{s+1}}\bigr)
 \le\eta D_{s,K}+C_\eta E_{s,K}.
\]
Thus, after combining the momentum terms and relabelling \(\eta\),
\begin{equation}\label{eq:3.2.12}
 \frac{1}{2}\frac{\mathrm d}{\mathrm dt}\|v\|_{H^s}^2+D_{s,K}
 -g\beta\sum_{|\alpha|\le s}\int_\Omega J_K\partial^\alpha\theta\,\partial^\alpha w
 \le\eta D_{s,K}+C_\eta(1+E_{s,K})^2.
\end{equation}

We next use \eqref{eq:2.3} to write the temperature equation as
\(\partial_t\theta+J_K(v\partial_x\theta+wR)=0\). For
\(h_\alpha=\partial^\alpha\theta\), the Leibniz expansion is
\begin{equation}\label{eq:3.2.13}
 \partial^\alpha(v\partial_x\theta+wR)
 =v\partial_xh_\alpha+w\partial_zh_\alpha
                 +R\partial^\alpha w+\mathcal R_\alpha,
\end{equation}
where \(\mathcal R_0=-w\partial_z\theta\), and, for \(|\alpha|\ge1\),
\begin{equation}\label{eq:3.2.14}
 \mathcal R_\alpha=[\partial^\alpha,v]\partial_x\theta+
 \sum_{\substack{0<\nu\le\alpha\\\nu\ne\alpha}}
 \binom\alpha\nu(\partial^\nu w)\partial^{\alpha-\nu}R.
\end{equation}
The mixed sum contains at most \(s-1\) derivatives of each factor.
Using \eqref{eq:2.5} with \(f=w\), \(g=R-\gamma\), and
\eqref{eq:2.15} for the horizontal commutator, we obtain
\[
\begin{aligned}
 \sum_{|\alpha|\le s}\|\mathcal R_\alpha\|_2\le{}&
 C\bigl(\|v\|_{H^s}+\|\theta\|_{H^s}
       +\|\nabla w\|_\infty\|R-\gamma\|_{H^{s-1}}\\
 &\hspace{2.5em}+\|\nabla R\|_\infty\|w\|_{H^{s-1}}\bigr)
 \le CE_{s,K}^{\frac{1}{2}}.
\end{aligned}
\]
The zeroth-order remainder obeys the same bound by
\(\|w\partial_z\theta\|_2\le\|w\|_\infty\|\theta\|_{H^1}\).
Testing the temperature equation by \(g\beta qh_\alpha\) gives
\begin{equation}\label{eq:3.2.15}
\begin{aligned}
 &\frac{g\beta}{2}\frac{\mathrm d}{\mathrm dt}\int_\Omega q|h_\alpha|^2
 +g\beta\int_\Omega qJ_K(R\partial^\alpha w)h_\alpha\\
 ={}&-g\beta\int_\Omega qJ_K(v\partial_xh_\alpha+w\partial_zh_\alpha)h_\alpha
      -g\beta\int_\Omega qJ_K\mathcal R_\alpha h_\alpha
      +\frac{g\beta}{2}\int_\Omega\partial_tq\,|h_\alpha|^2.
\end{aligned}
\end{equation}
The assumptions imply \(\|q\|_{W^{1,\infty}}\le C\).
For \((a,j)=(v,x)\) or \((w,z)\), use the exact identity
\[
 qJ_K(a\partial_jh)=J_K(aq\partial_jh)-[J_K,q](a\partial_jh).
\]
Lemma~\ref{lem:friedrichs-commutator} then yields
\[
 \left|\int_\Omega qJ_K(a\partial_jh)h\right|
 \le C\bigl(\|\nabla(aq)\|_\infty
       +\|\nabla q\|_\infty\|a\|_{W^{1,\infty}}\bigr)\|h\|_2^2.
\]
Together with the remainder bound, this controls the first two terms
on the right-hand side of \eqref{eq:3.2.15}, summed over \(\alpha\), by
\(CE_{s,K}\).

The derivative of the weight is estimated from the approximate
temperature equation itself. By \eqref{eq:2.14} and \eqref{eq:3.2.3},
\[
 \|\partial_t\theta\|_{H^3}
 \le C\bigl(\|v\|_{H^3}+\|\theta\|_{H^4}+\|w\|_{H^3}\bigr)
 \le CE_{s,K}^{\frac{1}{2}}.
\]
Since \(\partial_tq=-R^{-2}\partial_z\partial_t\theta\),
\eqref{eq:2.18} with \(r=2\), \(j=0\) implies
\[
 \|\partial_tq\|_\infty\le C\|\partial_t\theta\|_{H^3},
 \qquad
 \sum_{|\alpha|\le s}\left|\int_\Omega\partial_tq\,|h_\alpha|^2\right|
 \le CE_{s,K}^{\frac{3}{2}}.
\]
Only four derivatives of \(\theta\) are used here, so the estimate
holds at the lowest index \(s=4\).

Finally, the principal coupling has the decomposition
\begin{equation}\label{eq:3.2.16}
 R^{-1}J_K(R\partial^\alpha w)
 =J_K\partial^\alpha w+R^{-1}[J_K,R]\partial^\alpha w.
\end{equation}
By self-adjointness, the first term cancels the pressure coupling
in \eqref{eq:3.2.12} at every finite cutoff:
\[
 -\int_\Omega J_Kh_\alpha\,\partial^\alpha w
 +\int_\Omega h_\alpha J_K\partial^\alpha w=0.
\]
For the remaining term, \eqref{eq:3.2.4} and the diagnostic estimates give
\[
\begin{aligned}
 \sum_{|\alpha|\le s}\left|\int_\Omega
       q[J_K,R]\partial^\alpha w\,h_\alpha\right|
 \le{}&C\|\nabla R\|_\infty
       \sum_{|\alpha|\le s}\|\partial^\alpha w\|_{H^{-1}}\|h_\alpha\|_2\\
 \le{}&C\|w\|_{H^{s-1}}\|\theta\|_{H^s}
 \le CE_{s,K}.
\end{aligned}
\]
Adding the momentum and temperature identities, choosing \(\eta\)
sufficiently small, and using
\(E_{s,K}^{\frac{3}{2}}\le(1+E_{s,K})^2\), we conclude that
\[
 \frac{\mathrm d}{\mathrm dt}E_{s,K}+cD_{s,K}
 \le C(1+E_{s,K})^2.
\]
Integration and \eqref{eq:3.2.2} prove \eqref{eq:3.2.10}.
The additional derivative in the momentum estimate is controlled by the
velocity dissipation. No \(H^{s+1}\)-norm of the temperature is required.
\end{proof}

The weighted estimate yields a common lifespan once the energy and the
lower stratification bound are controlled simultaneously. The following
proposition supplies the uniform bounds used in the compactness argument.

\begin{proposition}\label{prop:uniform-friedrichs}
There exist \(K_0\ge1\) and \(T_*,C_*>0\), depending only on \(s\),
an upper bound for \(\|U_0\|_{H^s\times H^s}\), \(\kappa_0^{-1}\),
the periods, the physical parameters, and \(\chi\), such that
\(T_K^{\max}>T_*\) for every \(K\ge K_0\). On \([0,T_*]\times\Omega\),
\begin{equation}\label{eq:3.3.1}
 \sup_{0\le t\le T_*}E_{s,K}(t)
       +\int_0^{T_*}\|v_K\|_{H^{s+1}}^2\,\mathrm dt
 \le C_*,\qquad R_K\ge\frac{\kappa_0}{4}.
\end{equation}
\end{proposition}
\begin{proof}
Choose \(K_0\) as in \eqref{eq:3.1.2}, and take \(E_0\ge1\) with
\(E_{s,K}(0)\le E_0\) for every \(K\ge K_0\). Such a bound depends
only on the stated parameters. Consider the maximal interval starting
at zero on which
\(R_K>\frac{\kappa_0}{4}\) and \(E_{s,K}<2E_0\).

We first recover the unweighted norm without assuming an upper bound
for \(R_K\). By \eqref{eq:2.13}, the definition of \(E_{s,K}\), and
\eqref{eq:2.18},
\[
 \|\theta_K\|_{H^s}^2
 \le CE_0\|R_K\|_\infty
 \le CE_0\bigl(\gamma+C_s\|\theta_K\|_{H^s}\bigr)
 \le\frac{1}{2}\|\theta_K\|_{H^s}^2+C(1+E_0)^2.
\]
Hence \(\|U_K\|_{H^s\times H^s}\le C(1+E_0)\).
Choose \(M\) larger than the resulting \(W^{2,\infty}\)-bound from
\eqref{eq:2.18}. This choice is independent of \(K\), so
Lemma~\ref{lem:weighted-energy} applies with
\(\kappa=\frac{\kappa_0}{4}\) throughout the bootstrap interval.
In particular,
\[
 E_{s,K}(t)\le E_0+C_Et(1+2E_0)^2.
\]
The temperature equation, \eqref{eq:2.14}, and \eqref{eq:3.2.3} also imply
\[
 \|\partial_t\theta_K\|_{H^{s-1}}
 \le C\bigl(\|v_K\|_{H^{s-1}}+\|\theta_K\|_{H^s}
                      +\|w_K\|_{H^{s-1}}\bigr)\le C_\theta.
\]
Since \(s-2\ge2\), \eqref{eq:2.18} and \eqref{eq:3.1.2} yield
\[
 R_K(t)\ge R_{0,K}-\|\partial_z(\theta_K(t)-\theta_K(0))\|_\infty
 \ge\frac{\kappa_0}{2}-C_sC_\theta t.
\]
Taking \(C_E,C_\theta\ge1\), choose
\[
 T_*:=\min\left\{1,\frac{E_0}{2C_E(1+2E_0)^2},
                      \frac{\kappa_0}{8C_sC_\theta}\right\}.
\]
For \(t\le T_*\) in the bootstrap interval, these estimates improve the
bounds to \(E_{s,K}(t)\le\frac{3E_0}{2}\) and
\(R_K(t)\ge\frac{3\kappa_0}{8}\). Continuity therefore prevents the
bootstrap interval from ending before
\(\min\{T_*,T_K^{\max}\}\). If \(T_K^{\max}\le T_*\), the uniform
\(H^s\)-bound keeps the finite-dimensional solution bounded up to its
maximal time. The ODE continuation theorem then gives an extension,
a contradiction. Thus \(T_K^{\max}>T_*\), and \eqref{eq:3.2.10}
proves \eqref{eq:3.3.1}.
\end{proof}

%=====================================================================
% End of Section 3
%=====================================================================

%=====================================================================
% Section 4. Local existence, stability, and continuation
%=====================================================================

\section{Local existence, stability, and continuation}
\label{sec:existence}

Fix an integer \(s\ge4\). We first construct a local strong solution
from the approximations of Section~\ref{sec:friedrichs}, then prove
lower-order stability and continuation. Since the temperature has no
diffusion, its strong time continuity is established separately.
Throughout this section, \(w,p\) are reconstructed by
\eqref{eq:1.2.1} and \eqref{eq:1.2.3}. Integer-order norms in energy
identities are taken in the derivative form of \eqref{eq:2.13}.

\subsection{Local existence by compactness}

We first record the convergence properties of the Fourier cutoffs
and the associated weighted energy identity.

\begin{lemma}\label{lem:regularized-identities}
Let \(J_L\) be defined by \eqref{eq:3.1.1}. For every \(r\in\mathbb R\),
\(J_Lf\to f\) in \(C([0,T];H^r)\) whenever
\(f\in C([0,T];H^r)\), and in \(L^p(0,T;H^r)\) whenever
\(f\in L^p(0,T;H^r)\), \(1\le p<\infty\).
If \(a\in L^\infty(0,T;W^{1,\infty})\) and
\(f\in L^p(0,T;H^{-1})\), then
\begin{equation}\label{eq:4.1.1}
 [J_L,a]f\longrightarrow0\qquad\text{in }L^p(0,T;L^2),\quad 1\le p<\infty.
\end{equation}
The corresponding time-independent convergence also holds.
Let \(m\ge0\) be an integer, \(q\in W^{1,1}(0,T;L^\infty)\), and
\(h\in L^\infty(0,T;H^m)\cap W^{1,1}(0,T;H^{m-1})\).
Then \(E_L[q,h]:=\sum_{|\alpha|\le m}\int_\Omega
q|J_L\partial^\alpha h|^2\) is absolutely continuous, with
\begin{equation}\label{eq:4.1.2}
 \frac{\mathrm d}{\mathrm dt}E_L[q,h]=2\sum_{|\alpha|\le m}\int_\Omega qJ_L\partial^\alpha\partial_th\,J_L\partial^\alpha h
 +\sum_{|\alpha|\le m}\int_\Omega\partial_tq|J_L\partial^\alpha h|^2
\end{equation}
for almost every time. If also \(h\in C([0,T];H^m)\), then
\(E_L[q,h]\to\sum_{|\alpha|\le m}\int_\Omega q|\partial^\alpha h|^2\)
uniformly on \([0,T]\).
\end{lemma}
\begin{proof}
Uniform boundedness of \(J_L\) and strong convergence on \(H^r\)
imply uniform convergence on the compact range of any continuous
\(H^r\)-valued trajectory. The \(L^p\)-convergence follows by dominated
convergence. For \eqref{eq:4.1.1}, choose
\(f_j\in L_t^pL^2\) converging to \(f\) in \(L_t^pH^{-1}\).
By \eqref{eq:3.2.4},
\[
 \|[J_L,a](f-f_j)\|_{L_t^pL^2}
 \le C\|\nabla a\|_{L_t^\infty L^\infty}
          \|f-f_j\|_{L_t^pH^{-1}},
\]
uniformly in \(L\). For fixed \(j\), the identity
\[
 [J_L,a]f_j=(J_L-I)(af_j)-a(J_L-I)f_j
\]
and the first assertion imply convergence to zero in \(L_t^pL^2\).
Taking \(L\to\infty\) and then \(j\to\infty\) proves
\eqref{eq:4.1.1}. The same argument applies without the time variable.

For fixed \(L\), finite rank implies
\(J_L\partial^\alpha h\in W^{1,1}(0,T;L^2)\).
The product rule therefore proves \eqref{eq:4.1.2}.
If \(h\in C_tH^m\), the first assertion and \eqref{eq:2.13} yield
\[
 \sup_{t\in[0,T]}\left|E_L[q,h](t)
       -\sum_{|\alpha|\le m}\int_\Omega q|\partial^\alpha h|^2\right|
 \le C\|q\|_{L_t^\infty L^\infty}
       \|(J_L-I)h\|_{C_tH^m}\|h\|_{C_tH^m}\longrightarrow0.
\]
\end{proof}

We next recover strong time continuity of the temperature from its
equation, without assuming this continuity in the energy argument.

\begin{proposition}\label{prop:temperature-continuity}
Let
\(v\in C([0,T];H^s)\cap L^2(0,T;H^{s+1})\) and
\(\theta\in L^\infty(0,T;H^s)\cap W^{1,\infty}(0,T;H^{s-1})\).
Assume that \(\bar v=\bar\theta=0\),
\(R:=\gamma+\partial_z\theta\ge\kappa>0\), and the temperature
equation in \eqref{eq:1.1.2} holds in distributions, with \(w\)
given by \eqref{eq:1.2.1}.
Then \(\theta\in C([0,T];H^s)\), and its \(R^{-1}\)-weighted
\(H^s\)-energy is continuous on \([0,T]\).
\end{proposition}
\begin{proof}
The time regularity supplies a representative in \(C_tH^{s-1}\).
At any fixed time, approximate by times at which the essential
\(H^s\)-bound holds. Weak compactness in \(H^s\) and strong
continuity in \(H^{s-1}\) identify the weak limit with this
representative, so the same \(H^s\)-bound holds at every time.
Testing against finite Fourier sums and then using density yields
\(\theta\in C_w([0,T];H^s)\).

Set \(q=R^{-1}\). The assumed regularity of \(\theta\), together with
\eqref{eq:2.18} and \eqref{eq:2.20}, implies that
\(q\in C_tL^\infty\cap L_t^\infty W^{1,\infty}
\cap W_t^{1,\infty}L^\infty\), with
\(0<c\le q\le\kappa^{-1}\).
The diagnostic estimates \eqref{eq:3.2.3} also yield
\begin{equation}\label{eq:4.1.3}
 \|w\|_{L_t^\infty(H^{s-1}\cap W^{1,\infty})}\le C.
\end{equation}
In particular,
\(q_t=-R^{-2}\partial_z\theta_t\), and
\(\|q_t\|_{L_t^\infty L^\infty}
\le C\|\theta_t\|_{L_t^\infty H^{s-1}}\), since \(s-2>1\).
All constants in this proof are independent of \(N\).

For \(h_\alpha=\partial^\alpha\theta\), define
\(E_{\theta,N}:=g\beta\sum_{|\alpha|\le s}\int_\Omega
q|J_Nh_\alpha|^2\).
The differentiated temperature equation has the form
\[
 \partial_th_\alpha+v\partial_xh_\alpha+w\partial_zh_\alpha
       +R\partial^\alpha w+\mathcal R_\alpha=0
       \qquad\text{in }H^{-1},
\]
where \(\mathcal R_\alpha\) is given by
\eqref{eq:3.2.13}--\eqref{eq:3.2.14}.
Equations \eqref{eq:2.5}, \eqref{eq:2.15}, and
\eqref{eq:3.2.3} imply
\(\sum_{|\alpha|\le s}\|\mathcal R_\alpha\|_2\le C\).
Applying \(J_N\) to this equation and using
Lemma~\ref{lem:regularized-identities} is legitimate before strong
\(H^s\)-continuity has been established.

For \((a,j)=(v,x)\) or \((w,z)\), the transport pairing is
\[
 \int_\Omega qJ_N(a\partial_jh_\alpha)J_Nh_\alpha
 =-\frac{1}{2}\int_\Omega\partial_j(qa)|J_Nh_\alpha|^2
 +\int_\Omega q[J_N,a]\partial_jh_\alpha\,J_Nh_\alpha.
\]
It is bounded uniformly in \(N\) by \eqref{eq:3.2.4} and
\eqref{eq:4.1.3}. The same is true of the remainder and weight-derivative
terms. For the principal coupling, \eqref{eq:3.2.16} yields
\[
 \int_\Omega qJ_N(R\partial^\alpha w)J_Nh_\alpha
 =\int_\Omega J_N\partial^\alpha w^0\,J_Nh_\alpha
 +\int_\Omega q[J_N,R]\partial^\alpha w\,J_Nh_\alpha.
\]
Here the homogeneous component has zero pairing, because
\(\overline{J_Nh_\alpha}=0\).
The first term is bounded by \(C\|v\|_{H^{s+1}}\) using
\eqref{eq:2.10}. The second is bounded uniformly by
\eqref{eq:3.2.4}, since
\(\sum_{|\alpha|\le s}\|\partial^\alpha w\|_{H^{-1}}
\le C\|w\|_{H^{s-1}}\).
Thus
\[
|\partial_tE_{\theta,N}|\le C(1+\|v\|_{H^{s+1}}^2)
\]
almost everywhere, and
\begin{equation}\label{eq:4.1.4}
 |E_{\theta,N}(t)-E_{\theta,N}(t')|
 \le C\int_{\min\{t,t'\}}^{\max\{t,t'\}}
       (1+\|v(\rho)\|_{H^{s+1}}^2)\,\mathrm d\rho.
\end{equation}
At each fixed time, \(J_N\theta(t)\to\theta(t)\) in \(H^s\).
Passing to the limit in \eqref{eq:4.1.4} therefore proves continuity of
\(E_\theta(t):=g\beta\sum_{|\alpha|\le s}\int_\Omega
q(t)|\partial^\alpha\theta(t)|^2\).
If \(t_n\to t\), then \(\theta(t_n)\rightharpoonup\theta(t)\) in
\(H^s\) and \(q(t_n)\to q(t)\) in \(L^\infty\). Consequently,
\[
 \sum_{|\alpha|\le s}\int_\Omega q(t)|\partial^\alpha\theta(t_n)|^2
 \longrightarrow
 \sum_{|\alpha|\le s}\int_\Omega q(t)|\partial^\alpha\theta(t)|^2.
\]
By \eqref{eq:4.1.3} and \eqref{eq:2.13}, this is convergence of norms
in a fixed equivalent Hilbert norm on \(H^s\).
Together with weak convergence, it yields
\(\theta(t_n)\to\theta(t)\) in \(H^s\).
\end{proof}

We now pass to the limit in the approximate equations. We retain the
dependence of the local time and bounds on the initial norm and the
stratification lower bound, since the same constants will be used at
restart times.

\begin{proposition}\label{prop:local-existence}
For every \(M\ge1\) and \(\kappa>0\), there exist
\(\tau_*(M,\kappa)>0\) and \(M_*(M,\kappa)\ge1\) such that every
\(U_0\in\mathcal X_\kappa^s\) with
\(\|U_0\|_{H^s\times H^s}\le M\) admits an \(H^s\)-strong solution
on \([0,\tau_*]\) satisfying
\begin{equation}\label{eq:4.1.5}
 \sup_{0\le t\le\tau_*}\|U(t)\|_{H^s\times H^s}
       +\|v\|_{L^2(0,\tau_*;H^{s+1})}\le M_*,
 \qquad R\ge\frac{\kappa}{4}.
\end{equation}
The constants depend only on \(M,\kappa,s\), the periods, and the
physical parameters.
\end{proposition}
\begin{proof}
By Proposition~\ref{prop:uniform-friedrichs}, the approximations satisfy
uniform bounds in \(L_t^\infty(H^s\times H^s)\), with
\(v_K\) bounded in \(L_t^2H^{s+1}\) and \(R_K\ge\frac{\kappa}{4}\)
on a common interval \([0,\tau_*]\).
The equations, \eqref{eq:2.10}, \eqref{eq:2.17},
\eqref{eq:2.21}, and \eqref{eq:3.2.3} also imply
\[
 \|\partial_tv_K\|_{L_t^2H^{s-1}}
       +\|\partial_t\theta_K\|_{L_t^\infty H^{s-1}}\le C.
\]
In particular,
\[
\begin{aligned}
 \|v_K(t)-v_K(t')\|_{H^{s-1}}\le{}& C|t-t'|^{\frac{1}{2}},\\
 \|\theta_K(t)-\theta_K(t')\|_{H^{s-1}}\le{}& C|t-t'|.
\end{aligned}
\]
The compact embedding \(H^s\hookrightarrow H^{s-1}\) and the
Arzel\`a--Ascoli theorem yield compactness in \(C_tH^{s-1}\), as in
the criteria of \cite{Simon1987}. After extracting a subsequence,
\begin{equation}\label{eq:4.1.6}
 U_K\longrightarrow U\quad\text{in }C_t(H^{s-1}\times H^{s-1}),
 \qquad v_K\rightharpoonup v\quad\text{in }L_t^2H^{s+1}.
\end{equation}
The sequence also converges weakly-* in the bounded
\(L_t^\infty H^s\) spaces, and the time-derivative bounds pass to the
limit. The mean constraints are preserved. Moreover,
\(\partial_z\theta_K\to\partial_z\theta\) in
\(C_tH^{s-2}\hookrightarrow C_tL^\infty\) by \eqref{eq:2.18}, so
\(R\ge\frac{\kappa}{4}\).

The diagnostic formulas and \eqref{eq:2.10}, \eqref{eq:2.17},
\eqref{eq:2.21} yield
\(w_K\to w\) in \(C_tH^{s-2}\) and
\(p_K\to p\) in \(C_tH^{s-1}\).
Since \(s-2>1\), the algebra estimate \eqref{eq:2.17} implies
convergence of the uncut nonlinear terms in \(C_tH^{s-2}\).
The projected temperature nonlinearity is treated in its equivalent
form \(v_K\partial_x\theta_K+w_KR_K\), given by \eqref{eq:2.3}.
For any such term \(G_K\to G\),
\[
 \|J_KG_K-G\|_{C_tH^{s-2}}
 \le C\|G_K-G\|_{C_tH^{s-2}}+
       \|(J_K-I)G\|_{C_tH^{s-2}}\longrightarrow0,
\]
by Lemma~\ref{lem:regularized-identities}.
The viscous terms converge weakly in \(L_t^2H^{s-1}\).
Hence \eqref{eq:1.1.2} holds in distributions, and
\(U(0)=U_0\) follows from \eqref{eq:4.1.6} and
\(J_KU_0\to U_0\).

The Lions--Magenes lemma for
\(H^{s+1}\hookrightarrow H^s\hookrightarrow H^{s-1}\)
\cite{LionsMagenes1972} implies \(v\in C_tH^s\).
Proposition~\ref{prop:temperature-continuity} then yields
\(\theta\in C_tH^s\).
The diagnostic formulas and \eqref{eq:2.10}, \eqref{eq:2.17},
\eqref{eq:2.21} yield \(w\in C_tH^{s-1}\) and \(p\in C_tH^s\).
The temperature equation now implies
\(\partial_t\theta\in C_tH^{s-1}\).
Thus the limit is a strong solution in the sense of
Definition~\ref{def:strong-solution}.
Lower semicontinuity and the preceding continuity establish
\eqref{eq:4.1.5}. All bounds depend only on the initial norm bound and
stratification lower bound, as asserted.
\end{proof}

\subsection{Lower-order stability and continuation}

We next establish a stability estimate one derivative below the
solution regularity.
Let \(\sigma\ge4\) be an integer, put \(r=\sigma-1\), and let
\(U_i=(v_i,\theta_i)\), \(i=1,2\), be two \(H^\sigma\)-strong
solutions on \([0,T]\) such that
\begin{equation}\label{eq:4.2.1}
 \sup_{0\le t\le T}\|U_i(t)\|_{H^\sigma\times H^\sigma}\le M,
 \qquad R_i:=\gamma+\partial_z\theta_i\ge\kappa>0,
 \qquad i=1,2.
\end{equation}
Write \(\delta U=U_1-U_2\), use the same convention for the diagnostic
variables, and set
\(\mathcal T_i:=\partial_t+v_i\partial_x+w_i\partial_z\).
Subtracting the equations yields
\begin{equation}\label{eq:4.2.2}
\left\{
\begin{aligned}
 &\mathcal T_1\delta v-\mu_x\partial_{xx}\delta v-\mu_z\partial_{zz}\delta v
       +\partial_x\delta p=-\delta v\partial_xv_2-\delta w\partial_zv_2,\\
 &\mathcal T_1\delta\theta+R_2\delta w=-\delta v\partial_x\theta_2.
\end{aligned}
\right.
\end{equation}
The diagnostic difference is
\[
 \delta w=\delta w^0+\delta\bar w,
 \quad\delta w^0=-\partial_z^{-1}\partial_x\delta v,
 \quad\delta\bar w=-\frac{1}{\gamma}\partial_x
       \overline{\delta v\theta_1+v_2\delta\theta}.
\]
Equations \eqref{eq:2.10}, \eqref{eq:2.17}, \eqref{eq:2.18}, and
\eqref{eq:2.21} imply
\begin{equation}\label{eq:4.2.3}
 \|\delta w\|_{H^{r-1}}+\|\delta w\|_\infty
       \le C\|\delta U\|_{H^r\times H^r},
 \qquad\|\delta w^0\|_{H^r}\le C\|\delta v\|_{H^{r+1}}.
\end{equation}
Set \(X_r:=\|\delta U\|_{H^r\times H^r}\) and
\[
 E_{\delta,r}:=\|\delta v\|_{H^r}^2+
 g\beta\sum_{|\alpha|\le r}\int_\Omega R_2^{-1}
       |\partial^\alpha\delta\theta|^2.
\]
By \eqref{eq:2.13}, \eqref{eq:2.18}, and \eqref{eq:4.2.1},
\(C^{-1}X_r(t)^2\le E_{\delta,r}(t)\le CX_r(t)^2\)
uniformly for \(t\in[0,T]\).
The following proposition establishes a stability estimate for the
difference of two solutions. Uniqueness follows when their initial
data coincide.

\begin{proposition}\label{prop:lower-order-stability}
Under \eqref{eq:4.2.1}, for every \(0\le t\le T\),
\begin{equation}\label{eq:4.2.4}
 E_{\delta,r}(t)+c\int_0^t\|\delta v\|_{H^{r+1}}^2\,\mathrm d\rho
 \le E_{\delta,r}(0)+C\int_0^t E_{\delta,r}\,\mathrm d\rho.
\end{equation}
Consequently,
\begin{equation}\label{eq:4.2.5}
 \|\delta U\|_{C([0,T];H^r\times H^r)}^2
       +\|\delta v\|_{L^2(0,T;H^{r+1})}^2
 \le C_T\|\delta U(0)\|_{H^r\times H^r}^2.
\end{equation}
All constants are uniform for solutions satisfying \eqref{eq:4.2.1},
and \(c,C\) are independent of \(T\).
In particular, the two solutions coincide on \([0,T]\) if their
initial data agree.
\end{proposition}

\begin{proof}
We apply \(\partial^\alpha\) to the first equation in
\eqref{eq:4.2.2}, test against \(\partial^\alpha\delta v\), and sum
over \(|\alpha|\le r\).
The hydrostatic relation implies
\[
 \sum_{|\alpha|\le r}\int_\Omega
       \partial^\alpha\partial_x\delta p\,\partial^\alpha\delta v
 =-g\beta\sum_{|\alpha|\le r}\int_\Omega
       \partial^\alpha\delta\theta\,\partial^\alpha\delta w.
\]
This is \eqref{eq:3.2.11} without cutoff, applied to the difference.
The homogeneous pressure and velocity components have zero pairing
with \(\partial^\alpha\delta v\) and
\(\partial^\alpha\delta\theta\), respectively.
The principal transport vanishes because
\(\partial_xv_1+\partial_zw_1=0\).
By \eqref{eq:2.15}, \eqref{eq:2.18}, and \eqref{eq:3.2.3},
\[
 \sum_{|\alpha|\le r}\bigl(
 \|[\partial^\alpha,v_1]\partial_x\delta v\|_2+
 \|[\partial^\alpha,w_1]\partial_z\delta v\|_2\bigr)\le CX_r.
\]
Indeed, \(w_1\in H^r\cap W^{1,\infty}\) is controlled by the
\(H^{r+1}\)-bound for \(U_1\), and
\(\|\nabla\delta v\|_\infty\le CX_r\) since \(r\ge3\).

For the forcing
\(G:=\delta v\partial_xv_2+\delta w\partial_zv_2\),
\eqref{eq:2.17} and \eqref{eq:4.2.3} imply
\[
 \|G\|_{H^{r-1}}\le C\bigl(
 \|\delta v\|_{H^{r-1}}\|\partial_xv_2\|_{H^{r-1}}+
 \|\delta w\|_{H^{r-1}}\|\partial_zv_2\|_{H^{r-1}}\bigr)
 \le CX_r.
\]
After one integration by parts in each positive-order pairing,
\[
 \sum_{|\alpha|\le r}\left|
 \langle\partial^\alpha G,\partial^\alpha\delta v\rangle
 \right|\le C\|G\|_{H^{r-1}}\|\delta v\|_{H^{r+1}}
 \le \eta\|\delta v\|_{H^{r+1}}^2+C_\eta X_r^2.
\]
Here the brackets denote the \(H^{-1}\)--\(H^1\) duality extending the
\(L^2\) pairing. This estimate treats the full diagnostic difference
\(\delta w\), including its homogeneous component, without an
\(H^r\)-bound for that difference.
The Poincar\'e estimate \eqref{eq:2.12} and a sufficiently small choice
of \(\eta\) now yield
\begin{equation}\label{eq:4.2.6}
\begin{aligned}
 \frac{1}{2}\frac{\mathrm d}{\mathrm dt}\|\delta v\|_{H^r}^2
 +c\|\delta v\|_{H^{r+1}}^2
 -g\beta\sum_{|\alpha|\le r}\int_\Omega
       \partial^\alpha\delta\theta\,\partial^\alpha\delta w
 \le{}&CE_{\delta,r}.
\end{aligned}
\end{equation}
Applying \(\partial^\alpha\) to the second equation in
\eqref{eq:4.2.2}, for \(|\alpha|\le r\), yields
\[
 \mathcal T_1\partial^\alpha\delta\theta
 +R_2\partial^\alpha\delta w+\mathcal R_\alpha^\delta=0,
\]
where
\[
 \mathcal R_\alpha^\delta:=
 [\partial^\alpha,v_1]\partial_x\delta\theta
 +[\partial^\alpha,w_1]\partial_z\delta\theta
 +\partial^\alpha(\delta v\partial_x\theta_2)
 +[\partial^\alpha,R_2]\delta w.
\]
The first three terms have total \(L^2\)-norm at most \(CX_r\) by
\eqref{eq:2.14}--\eqref{eq:2.15} and \eqref{eq:2.18}.
In particular, only \(\|\nabla\delta\theta\|_\infty\le CX_r\)
is used, not a \(W^{2,\infty}\)-bound for the difference.
For \(|\alpha|\ge1\), expand the remaining commutator as
\[
 [\partial^\alpha,R_2]\delta w
 =(\partial^\alpha R_2)\delta w+
 \sum_{\substack{0<\nu\le\alpha\\\nu\ne\alpha}}
 \binom\alpha\nu(\partial^\nu R_2)
                         \partial^{\alpha-\nu}\delta w.
\]
The endpoint is bounded by
\(\|R_2\|_{H^r}\|\delta w\|_\infty\le CX_r\).
For each term in the sum, put
\(a=r-|\nu|\) and \(b=r-1-|\alpha-\nu|\).
Then \(a,b\ge0\) and
\(a+b=2r-1-|\alpha|\ge r-1>1\). Hence \eqref{eq:2.16} implies
\[
 \|(\partial^\nu R_2)\partial^{\alpha-\nu}\delta w\|_2
 \le C\|R_2\|_{H^r}\|\delta w\|_{H^{r-1}}\le CX_r.
\]
The zero-order commutator vanishes. Thus
\(\sum_{|\alpha|\le r}\|\mathcal R_\alpha^\delta\|_2\le CX_r\)
for every \(r\ge3\).

Set \(q_2=R_2^{-1}\). Differentiating the temperature equation for
\(U_2\) in \(z\) and using incompressibility yields
\[
 \mathcal T_1q_2=-q_2\partial_xv_2
 +q_2^2(\partial_zv_2)\partial_x\theta_2
 +(\delta v\partial_x+\delta w\partial_z)q_2.
\]
Equations \eqref{eq:2.18}, \eqref{eq:2.20}, and \eqref{eq:4.2.3}
therefore imply \(\|\mathcal T_1q_2\|_\infty\le C\).
Testing the differentiated temperature equation against
\(g\beta q_2\partial^\alpha\delta\theta\) and summing over
\(|\alpha|\le r\), we obtain
\[
\begin{aligned}
 &\frac{g\beta}{2}\frac{\mathrm d}{\mathrm dt}
       \sum_{|\alpha|\le r}\int_\Omega
       q_2|\partial^\alpha\delta\theta|^2
 +g\beta\sum_{|\alpha|\le r}\int_\Omega
       \partial^\alpha\delta w\,\partial^\alpha\delta\theta\\
 ={}&\frac{g\beta}{2}\sum_{|\alpha|\le r}\int_\Omega
       (\mathcal T_1q_2)|\partial^\alpha\delta\theta|^2
 -g\beta\sum_{|\alpha|\le r}\int_\Omega
       q_2\mathcal R_\alpha^\delta\,\partial^\alpha\delta\theta\\
 \le{}&CE_{\delta,r}.
\end{aligned}
\]
Adding this estimate to \eqref{eq:4.2.6} cancels the coupling and
yields \eqref{eq:4.2.4} after integration.
Gronwall's inequality and energy equivalence prove \eqref{eq:4.2.5}.
For two \(H^s\)-strong solutions with the same datum, compactness
of the time interval and strict stratification supply common
\(M,\kappa\). Applying \eqref{eq:4.2.5} with \(\sigma=s\) gives
\(\delta U=0\). The diagnostic formulas and the pressure
normalization then imply coincidence of the full solutions.
\end{proof}

The quantitative local existence result in
Proposition~\ref{prop:local-existence}, together with uniqueness,
yields the following result on maximal solutions and continuation.

\begin{proposition}\label{prop:maximal-solution}
Let \(s\ge4\) be an integer. Every \(U_0\in\mathcal X^s\) admits a
unique maximal \(H^s\)-strong solution \(U\) on
\([0,T_{\max}^{(s)}(U_0))\), where
\(T_{\max}^{(s)}(U_0)\in(0,\infty]\). Maximality is understood in the
strictly stratified class of Definition~\ref{def:strong-solution}.
If \(T_{\max}^{(s)}(U_0)<\infty\), then
\begin{equation}\label{eq:4.3.1}
 \limsup_{t\rightarrow T_{\max}^{(s)}(U_0)^{-}}
 \bigl(\|U(t)\|_{H^s\times H^s}+\|R(t)^{-1}\|_\infty\bigr)=\infty.
\end{equation}
For every integer \(\sigma\ge s\) and \(U_0\in\mathcal X^\sigma\),
the maximal solutions at the two regularity levels agree on their
common interval, and
\(T_{\max}^{(\sigma)}(U_0)\le T_{\max}^{(s)}(U_0)\).
The criterion \eqref{eq:4.3.1} also holds with \(s\) replaced by
\(\sigma\).
\end{proposition}

\begin{proof}
Define \(T_{\max}^{(s)}(U_0)\) as the supremum of all \(T>0\) for
which an \(H^s\)-strong solution with initial datum \(U_0\) exists on
\([0,T]\). This supremum is positive by
Proposition~\ref{prop:local-existence}. By
Proposition~\ref{prop:lower-order-stability}, any two such solutions
agree on their common interval. They therefore define a unique
solution on \([0,T_{\max}^{(s)}(U_0))\) whose restriction to every
compact subinterval is \(H^s\)-strong. An extension beyond
\(T_{\max}^{(s)}(U_0)\) would contradict the definition of the supremum.

Suppose that \(T_{\max}^{(s)}(U_0)<\infty\) and the limsup in
\eqref{eq:4.3.1} is finite. The resulting bounds near the terminal
time, together with the strong-solution bounds on compact subintervals,
imply that, for some \(M\ge1\) and \(\kappa>0\),
\[
 \sup_{0\le t<T_{\max}^{(s)}(U_0)}\|U(t)\|_{H^s\times H^s}\le M,
 \qquad
 \inf_{[0,T_{\max}^{(s)}(U_0))\times\Omega}R\ge\kappa.
\]
Choose \(t_0\in(0,T_{\max}^{(s)}(U_0))\) such that
\[
 T_{\max}^{(s)}(U_0)-t_0<\tau_*(M,\kappa).
\]
Proposition~\ref{prop:local-existence}, applied to \(U(t_0)\), yields
a restarted solution on \([t_0,t_0+\tau_*(M,\kappa)]\).
By Proposition~\ref{prop:lower-order-stability}, it agrees with \(U\)
on \([t_0,T_{\max}^{(s)}(U_0))\). Patching the solutions therefore
preserves the strong-solution regularity and strict stratification,
and yields a solution on \([0,t_0+\tau_*(M,\kappa)]\). Consequently,
\[
 T_{\max}^{(s)}(U_0)\ge t_0+\tau_*(M,\kappa)
                    >T_{\max}^{(s)}(U_0),
\]
a contradiction. This proves \eqref{eq:4.3.1}.

Finally, let \(\sigma\ge s\) be an integer and
\(U_0\in\mathcal X^\sigma\). Every \(H^\sigma\)-strong solution is
also \(H^s\)-strong on the same compact time interval. Taking the
suprema of the corresponding admissible existence times yields
\[
 T_{\max}^{(\sigma)}(U_0)\le T_{\max}^{(s)}(U_0).
\]
The two maximal solutions agree on every common compact interval by
Proposition~\ref{prop:lower-order-stability}. The continuation
criterion at level \(\sigma\) follows from the preceding argument
with \(s\) replaced by \(\sigma\).
\end{proof}

%=====================================================================
% End of Section 4
%=====================================================================

%=====================================================================
% Section 5. Bona--Smith continuity
%=====================================================================

\section{Bona--Smith continuity}
\label{sec:bona-smith}

Fix an integer \(s\ge4\). We prove continuous dependence in
\(H^s\times H^s\) on compact subintervals of the reference lifespan
by the Bona--Smith method \cite{BonaSmith1975}.
The argument uses a one-sided comparison in which higher-order norms
of the more regular solution multiply only lower-order norms of the
difference.

We write \(U(t;A)\) for the primary variables of the maximal solution
with initial datum \(A\), and \(T_{\max}(A)=T_{\max}^{(s)}(A)\).
In energy identities, we use the equivalent derivative norm
in \eqref{eq:2.13}.

\subsection{Smoothing and higher-order estimates}
\label{subsec:smoothing-lifespans}

For the initial data, set \(S_N=J_N\), where \(J_N\) is defined by
\eqref{eq:3.1.1}. The Fourier norm \eqref{eq:2.8} and the support
of the cutoff imply that, for \(r\in\mathbb R\), \(q\ge0\),
\(N\ge1\), and \(f\in H^r\),
\begin{equation}\label{eq:5.1.1}
 \|S_Nf\|_{H^{r+q}}\le C\langle N\rangle^q\|f\|_{H^r},\qquad
 \langle N\rangle^q\|(I-S_N)f\|_{H^{r-q}}\le C\|f\|_{H^r},
\end{equation}
where \(\langle N\rangle=(1+N^2)^{\frac{1}{2}}\) and \(C\) is
independent of \(N\). Since \(1-\chi(\frac{\cdot}{N})\) vanishes
on frequencies of length at most \(N\),
\[
 \langle N\rangle^{2q}\|(I-S_N)f\|_{H^{r-q}}^2
 \le CL_xL_z\sum_{\substack{(k,m)\in\mathfrak K\\ |(k,m)|>N}}
       (1+k^2+m^2)^r|\widehat f(k,m)|^2.
\]
Hence, for every \(f\in H^r\),
\begin{equation}\label{eq:5.1.2}
 \langle N\rangle^q\|(I-S_N)f\|_{H^{r-q}}\longrightarrow0
 \qquad\text{as }N\to\infty.
\end{equation}
If \(f_n\to f\) in \(H^r\), applying \eqref{eq:5.1.1} to
\(f_n-f\) and using \eqref{eq:5.1.2} yields
\begin{equation}\label{eq:5.1.3}
 \lim_{N\to\infty}\limsup_{n\to\infty}
 \langle N\rangle^q\|(I-S_N)f_n\|_{H^{r-q}}=0.
\end{equation}
For fixed \(N\), the first estimate in \eqref{eq:5.1.1} also
implies \(S_Nf_n\to S_Nf\) in \(H^\sigma\) for every \(\sigma\ge r\).

We next obtain a common existence interval for the original and
smoothed data. For \(U=(v,\theta)\) and
\(\widetilde U=(\widetilde v,\widetilde\theta)\) in \(H^s\times H^s\),
the embedding \eqref{eq:2.18} yields
\begin{equation}\label{eq:5.1.4}
 \inf_\Omega(\gamma+\partial_z\widetilde\theta)
 \ge\inf_\Omega(\gamma+\partial_z\theta)
       -C_s\|\widetilde U-U\|_{H^s\times H^s}.
\end{equation}
Let \(A_n\in\mathcal Z^s\) satisfy
\(A_n\to A\in\mathcal X_\kappa^s\) in \(H^s\times H^s\).
Since \(S_N\) preserves the mean constraints,
\(A,A_n,S_NA,S_NA_n\in\mathcal Z^s\). Moreover,
\[
 \|S_NA_n-A\|_{H^s\times H^s}
 \le C\|A_n-A\|_{H^s\times H^s}
       +\|(S_N-I)A\|_{H^s\times H^s}.
\]
Together with \eqref{eq:5.1.2} and \eqref{eq:5.1.4}, this implies
that the four initial data belong to
\(\mathcal X_{\frac{\kappa}{2}}^s\) and have a common
\(H^s\times H^s\)-bound for all sufficiently large \(n,N\).
The thresholds for \(n\) and \(N\) can be chosen independently.
Proposition~\ref{prop:local-existence} therefore provides a common
interval \([0,\tau_0]\) on which the corresponding solutions satisfy
\eqref{eq:4.1.5}, with uniform norm and dissipation bounds and a
uniform positive stratification lower bound.
The interval and these bounds depend only on the common initial
norm bound and \(\kappa\), with the remaining parameters fixed,
and are independent of \(n,N\).

For an integer \(\sigma\ge s\), define \(E_\sigma,D_\sigma\) by
\eqref{eq:3.2.1}, with \(s\) replaced by \(\sigma\) and the cutoff
subscript omitted. The following proposition establishes persistence
of higher regularity through an estimate whose constant is independent
of the higher-order size of the solution.

\begin{proposition}\label{prop:tame-persistence}
Let \(\sigma\ge s\ge4\) be integers and \(T>0\). Suppose that an
\(H^\sigma\)-strong solution \(U\) on \([0,T]\) satisfies
\begin{equation}\label{eq:5.2.1}
 \sup_{0\le t\le T}\|U(t)\|_{H^s\times H^s}\le M,
 \qquad \inf_{[0,T]\times\Omega}R\ge\kappa,
\end{equation}
for some \(M\ge1\) and \(\kappa>0\). Then
\begin{equation}\label{eq:5.2.2}
 \sup_{0\le t\le T}\|U(t)\|_{H^\sigma\times H^\sigma}^2
       +\int_0^T\|v\|_{H^{\sigma+1}}^2\,\mathrm dt
 \le C\|U(0)\|_{H^\sigma\times H^\sigma}^2.
\end{equation}
The constant \(C\) is independent of the higher-order norms of \(U\).
If an \(H^s\)-strong solution satisfies \eqref{eq:5.2.1} and
\(U(0)\in\mathcal X^\sigma\), then it is \(H^\sigma\)-strong on
\([0,T]\) and satisfies \eqref{eq:5.2.2}.
\end{proposition}

\begin{proof}
Write \(q=R^{-1}\). By \eqref{eq:2.18}, \eqref{eq:2.20}, and
\eqref{eq:3.2.3}, the bounds \eqref{eq:5.2.1} control
\(v,\theta\) in \(W^{2,\infty}\), \(w\) in \(W^{1,\infty}\),
and \(R,q\) in \(W^{1,\infty}\), with \(0<c\le q\le\kappa^{-1}\).
Since \(\bar v=0\), \eqref{eq:2.12} and \eqref{eq:2.13} yield
\[
 C^{-1}\|U\|_{H^\sigma\times H^\sigma}^2\le E_\sigma
 \le C\|U\|_{H^\sigma\times H^\sigma}^2,
 \qquad D_\sigma\ge c\|v\|_{H^{\sigma+1}}^2.
\]
The constants in the differential estimates below are independent of
\(T\) and the higher-order norms of \(U\).
Equations \eqref{eq:2.10}, \eqref{eq:2.14}, and \eqref{eq:2.21}
imply
\begin{equation}\label{eq:5.2.3}
 \|w\|_{H^{\sigma-1}}+\|R-\gamma\|_{H^{\sigma-1}}
 \le CE_\sigma^{\frac{1}{2}},\qquad
 \|w^0\|_{H^\sigma}\le C\|v\|_{H^{\sigma+1}}.
\end{equation}

For smooth solutions, \eqref{eq:3.2.7}, with \(m=\sigma\),
separates the horizontal endpoint from the momentum commutators.
By \eqref{eq:2.15}, \eqref{eq:3.2.6}, and \eqref{eq:5.2.3},
\[
 \sum_{|\alpha|\le\sigma}\bigl(
 \|[\partial^\alpha,v]\partial_xv\|_2+
 \|[\partial^\alpha,w^0]\partial_zv\|_2+
 \|\mathcal H_\alpha[\bar w,v]\|_2\bigr)
 \le C\bigl(E_\sigma^{\frac{1}{2}}+D_\sigma^{\frac{1}{2}}\bigr).
\]
Their energy pairings are bounded by
\(\eta D_\sigma+C_\eta E_\sigma\) for every \(\eta>0\), and the
principal transport vanishes by incompressibility.
Since \eqref{eq:2.14} and \eqref{eq:2.21} imply
\(\|\overline{v\theta}\|_{H_x^\sigma}\le CE_\sigma^{\frac{1}{2}}\),
the horizontal endpoint satisfies
\begin{equation}\label{eq:5.2.4}
 |\mathcal B_\sigma(\overline{v\theta};v,v)|
 \le CE_\sigma^{\frac{1}{2}}
       \bigl(E_\sigma^{\frac{1}{2}}+D_\sigma^{\frac{1}{2}}\bigr)
 \le\eta D_\sigma+C_\eta E_\sigma
\end{equation}
by \eqref{eq:3.2.8} and \eqref{eq:3.2.9}.

The temperature equation and incompressibility imply
\begin{equation}\label{eq:5.2.5}
 (\partial_t+v\partial_x+w\partial_z)q
 =-q\partial_xv+q^2(\partial_zv)\partial_x\theta.
\end{equation}
Its right-hand side is bounded in \(L^\infty\) by
\eqref{eq:5.2.1}.
The remainder in \eqref{eq:3.2.13}--\eqref{eq:3.2.14} satisfies
\[
 \sum_{|\alpha|\le\sigma}\|\mathcal R_\alpha\|_2
 \le C\bigl(\|v\|_{H^\sigma}+\|\theta\|_{H^\sigma}
                  +\|w\|_{H^{\sigma-1}}\bigr)
 \le CE_\sigma^{\frac{1}{2}}
\]
by \eqref{eq:2.5}, \eqref{eq:2.15}, and \eqref{eq:5.2.3},
including the zero-order term \(\mathcal R_0=-w\partial_z\theta\).
Testing the differentiated temperature equation against
\(g\beta q\partial^\alpha\theta\), summing over
\(|\alpha|\le\sigma\), and using \eqref{eq:5.2.5}, we obtain
\[
 \frac{g\beta}{2}\frac{\mathrm d}{\mathrm dt}
 \sum_{|\alpha|\le\sigma}\int_\Omega q|\partial^\alpha\theta|^2
 +g\beta\sum_{|\alpha|\le\sigma}\int_\Omega
       \partial^\alpha w\,\partial^\alpha\theta
 \le CE_\sigma.
\]
The coupling cancels the pressure term in the momentum estimate by
\eqref{eq:3.2.11} without cutoff. Adding the two estimates and absorbing
\eqref{eq:5.2.4} and the non-endpoint pairings, with \(\eta\)
sufficiently small, yields
\begin{equation}\label{eq:5.2.6}
 \frac{\mathrm d}{\mathrm dt}E_\sigma+cD_\sigma\le CE_\sigma.
\end{equation}

To justify the calculation for an \(H^\sigma\)-strong solution,
apply \(J_L\partial^\alpha\) to the original equations.
The weighted product rule \eqref{eq:4.1.2} applies with weights
\(1,q\), since \(q_t=-q^2\partial_z\theta_t\in C_tL^\infty\)
by \(\theta_t\in C_tH^{\sigma-1}\) and \eqref{eq:2.18}.
At each finite cutoff,
\begin{equation}\label{eq:5.2.7}
\begin{aligned}
 &\int_\Omega J_L\partial^\alpha\partial_xp\,J_L\partial^\alpha v
   +g\beta\int_\Omega qJ_L(R\partial^\alpha w)
                                      J_L\partial^\alpha\theta\\
 ={}&g\beta\int_\Omega q[J_L,R]\partial^\alpha w\,
                                      J_L\partial^\alpha\theta.
\end{aligned}
\end{equation}
The regularized homogeneous pressure and velocity terms have zero
pairing, since
\(\overline{J_L\partial^\alpha v}
=\overline{J_L\partial^\alpha\theta}=0\).
Since \(w\in C_tH^{\sigma-1}\), \eqref{eq:4.1.1} implies
\[
 [J_L,R]\partial^\alpha w\longrightarrow0
 \qquad\text{in }L_t^1L^2.
\]
Thus the right-hand side of \eqref{eq:5.2.7} vanishes in
\(L_t^1\). The same result, applied componentwise, gives
\[
 [J_L,v]\partial_x\partial^\alpha U+
 [J_L,w]\partial_z\partial^\alpha U\longrightarrow0
 \qquad\text{in }L_t^1(L^2\times L^2).
\]
These commutators vanish in the integrated energy pairings. The
principal transport terms are handled by incompressibility and
\eqref{eq:5.2.5}.

The regularized horizontal endpoint is
\(\mathcal B_\sigma(\overline{v\theta};v,J_L^2v)\).
By linearity in the last argument and \eqref{eq:3.2.9},
\[
 \|\mathcal B_\sigma(\overline{v\theta};v,(J_L^2-I)v)\|_{L_t^1}
 \le CT^{\frac{1}{2}}\|\overline{v\theta}\|_{C_tH_x^\sigma}
             \|(J_L^2-I)v\|_{L_t^2H^{\sigma+1}}\longrightarrow0,
\]
where convergence follows from Lemma~\ref{lem:regularized-identities}.
All non-endpoint remainders belong to \(L_t^2L^2\), so their
regularizations converge in that space. The same lemma gives
\(J_Lv\to v\) in \(L_t^2H^{\sigma+1}\) and uniform convergence
of the weighted energies. Passing to the limit in the integrated
identities therefore proves \eqref{eq:5.2.6} in distributions.
Gronwall's inequality yields
\[
 E_\sigma(t)+c\int_0^t D_\sigma(\rho)\,\mathrm d\rho
 \le e^{Ct}E_\sigma(0),\qquad 0\le t\le T.
\]
Energy equivalence and the lower bound for \(D_\sigma\) prove
\eqref{eq:5.2.2}.

Finally, let \(U\) be an \(H^s\)-strong solution on \([0,T]\)
satisfying \eqref{eq:5.2.1}, with \(U(0)\in\mathcal X^\sigma\).
By Proposition~\ref{prop:maximal-solution}, the maximal
\(H^\sigma\)-strong solution with the same datum agrees with \(U\)
on every common compact interval.
If \(T_{\max}^{(\sigma)}(U(0))\le T\), the preceding estimate on
these intervals implies
\[
 \sup_{0\le t<T_{\max}^{(\sigma)}(U(0))}
 \bigl(\|U(t)\|_{H^\sigma\times H^\sigma}+\|q(t)\|_\infty\bigr)
 \le Ce^{CT}\|U(0)\|_{H^\sigma\times H^\sigma}
                 +\kappa^{-1}<\infty.
\]
The constants are uniform in the terminal time of the compact interval.
This contradicts \eqref{eq:4.3.1} at level \(\sigma\). Hence
\(T<T_{\max}^{(\sigma)}(U(0))\), which proves persistence on
\([0,T]\) and \eqref{eq:5.2.2}.
\end{proof}

Let \(\sigma\ge s\) be an integer, and let \(A\) range over a
bounded subset of \(H^s\times H^s\), with
\(S_NA\in\mathcal X^\sigma\).
Suppose that the corresponding \(H^s\)-strong solutions exist on
a common interval \([0,T]\), with uniform \(H^s\)-bounds and a
uniform positive stratification lower bound.
Proposition~\ref{prop:tame-persistence} ensures that these solutions
are \(H^\sigma\)-strong throughout \([0,T]\).
Combining \eqref{eq:5.1.1} with \eqref{eq:5.2.2}, we obtain
\begin{equation}\label{eq:5.2.8}
 \|U(\cdot;S_NA)\|_{C_t(H^\sigma\times H^\sigma)}
       +\|v(\cdot;S_NA)\|_{L_t^2H^{\sigma+1}}
 \le C_\sigma\langle N\rangle^{\sigma-s},
\end{equation}
where \(C_\sigma\) is independent of \(A\) and \(N\).

\subsection{One-sided comparison estimates}

We first treat the vertical endpoint arising in the top-order estimates for the projected temperature difference equation.
The following lemma extends the corresponding weighted pairing to a continuous quadratic form on \(L^2\) and establishes convergence of the regularized pairings.

\begin{lemma}\label{lem:vertical-endpoint}
Let \(q\in W^{1,\infty}(\Omega)\),
\(b\in W^{1,\infty}(\mathbb T_x)\), and \(h\in L^2(\Omega)\)
be real-valued, with \(b\) identified with its \(z\)-independent
extension to \(\Omega\). Define
\(\mathcal Q_{q,b}(h):=\frac{1}{2}\int_\Omega(\partial_zq)b|h|^2\).
Then
\[
 |\mathcal Q_{q,b}(h)|\le\frac{1}{2}
 \|\partial_zq\|_\infty\|b\|_\infty\|h\|_2^2,
\]
and \(\mathcal Q_{q,b}(h)\) is jointly continuous in \((q,b,h)\)
for convergence in \(W^{1,\infty}\times L^\infty\times L^2\).
For \(h\in H^1(\Omega)\),
\(\mathcal Q_{q,b}(h)=-\int_\Omega qb(\partial_zh)h\).
With \(J_L\) defined by \eqref{eq:3.1.1}, one has
\begin{equation}\label{eq:5.3.1}
 -\int_\Omega qJ_L(b\partial_zh)J_Lh
 \longrightarrow\mathcal Q_{q,b}(h)\qquad\text{as }L\to\infty.
\end{equation}
If \(q\in L_t^\infty W^{1,\infty}(\Omega)\),
\(b\in L_t^\infty W^{1,\infty}(\mathbb T_x)\), and
\(h\in C_tL^2(\Omega)\), then the convergence in
\eqref{eq:5.3.1} holds in \(L^1(0,T)\).
\end{lemma}

\begin{proof}
H\"older's inequality and continuity of \(h\mapsto |h|^2\) from
\(L^2\) to \(L^1\) prove the bound and joint continuity.
Periodic integration by parts, using \(\partial_zb=0\), gives the
identity for \(h\in H^1\). By density, \(\mathcal Q_{q,b}\) is the
unique continuous extension of this quadratic pairing to \(L^2\).

For \(h\in L^2\), the identity
\(b\partial_zh=\partial_z(bh)\in H^{-1}\) makes the regularized
pairing well defined. Commutation of \(J_L\) with \(\partial_z\)
and the preceding integration-by-parts identity yield
\[
 -\int_\Omega qJ_L(b\partial_zh)J_Lh
 =\mathcal Q_{q,b}(J_Lh)
       -\int_\Omega q[J_L,b]\partial_zh\,J_Lh.
\]
The first term tends to \(\mathcal Q_{q,b}(h)\) by strong
\(L^2\)-convergence. For the second, the time-independent version
of \eqref{eq:4.1.1}, applied to \(\partial_zh\in H^{-1}\), gives
\([J_L,b]\partial_zh\to0\) in \(L^2\). This proves
\eqref{eq:5.3.1}.

In the time-dependent case, the same decomposition and H\"older's
inequality yield
\[
 \left\|-\int_\Omega qJ_L(b\partial_zh)J_Lh
                 -\mathcal Q_{q,b}(h)\right\|_{L_t^1}
 \le C\|h\|_{C_tL^2}\bigl(T\|(J_L-I)h\|_{C_tL^2}+\|[J_L,b]\partial_zh\|_{L_t^1L^2}\bigr),
\]
where \(C\) is independent of \(L\).
Lemma~\ref{lem:regularized-identities} gives
\((J_L-I)h\to0\) in \(C_tL^2\), and \eqref{eq:4.1.1} with
\(p=1\) gives \([J_L,b]\partial_zh\to0\) in \(L_t^1L^2\).
The right-hand side therefore tends to zero.
\end{proof}

Let \(U=(v,\theta)\) and \(U^N=(v^N,\theta^N)\) be two
\(H^s\)-strong solutions on \([0,\tau]\) satisfying
\begin{equation}\label{eq:5.3.2}
 \sup_{0\le t\le\tau}
 \bigl(\|U(t)\|_{H^s\times H^s}
       +\|U^N(t)\|_{H^s\times H^s}\bigr)\le M,
 \qquad R,R^N\ge\kappa>0,
\end{equation}
where the stratification bounds hold on \([0,\tau]\times\Omega\).
No relation between the initial data is assumed.
Write \(\delta U^N=U^N-U\),
\(\Delta R=R^N-R=\partial_z\delta\theta^N\), and
\(\delta w^0=w^{N,0}-w^0=-\partial_z^{-1}\partial_x\delta v^N\).
Set
\[
 F=v\partial_x\theta+w^0R,\qquad
 F^N=v^N\partial_x\theta^N+w^{N,0}R^N,\qquad
 \Delta F=F^N-F.
\]
We use \(q=R^{-1}\) and
\(h_\alpha=\partial^\alpha\delta\theta^N\) for \(|\alpha|\le s\),
and define
\[
 X_j^2=\|\delta v^N\|_{H^j}^2+\|\delta\theta^N\|_{H^j}^2,
 \qquad j\in\{s-1,s\}.
\]
The higher-order norms of \(U^N\) enter through
\[
 M_N(t)=1+\|U^N(t)\|_{H^{s+1}\times H^{s+1}}
                 +\|v^N(t)\|_{H^{s+2}}.
\]
The following estimate holds pointwise at times for which
\(M_N(t)<\infty\). Its constant is independent of \(N\) and the
higher-order norms of \(U^N\).

Using the quadratic form in Lemma~\ref{lem:vertical-endpoint},
define
\(\mathcal C_0^N:=\frac{1}{\gamma}
\mathcal Q_{q,\overline{F^N}}(h_0)\)
and, for \(1\le|\alpha|\le s\), set
\begin{equation}\label{eq:5.3.3}
\begin{aligned}
 \mathcal C_\alpha^N:={}&-\frac{1}{\gamma}
 \sum_{\substack{0<\nu\le\alpha\\\nu\ne\alpha}}\binom\alpha\nu
 \int_\Omega q(\partial^\nu R)
                  \overline{\partial^{\alpha-\nu}\Delta F}\,h_\alpha\\
 &-\frac{1}{\gamma}
 \sum_{\substack{\nu\le\alpha\\\nu\ne\alpha}}\binom\alpha\nu
 \int_\Omega q(\partial^\nu\Delta R)
                  \overline{\partial^{\alpha-\nu}F^N}\,h_\alpha\\
 &-\frac{1}{\gamma}\int_\Omega q(\partial^\alpha R^N)
                  \overline{\Delta F}\,h_\alpha
       +\frac{1}{\gamma}\mathcal Q_{q,\overline F}(h_\alpha).
\end{aligned}
\end{equation}
The quadratic-form terms give the continuous extensions of the
vertical endpoint pairings to \(L^2\).
The following lemma estimates the projection commutator and the
projection difference jointly.

\begin{lemma}\label{lem:projected-commutator}
Under \eqref{eq:5.3.2}, the quantities \(\mathcal C_\alpha^N\),
\(|\alpha|\le s\), are well defined whenever \(M_N(t)<\infty\), and
\begin{equation}\label{eq:5.3.4}
 \left|\sum_{|\alpha|\le s}\mathcal C_\alpha^N\right|
 \le CX_s^2+CM_N^2X_{s-1}^2.
\end{equation}
The constant \(C\) is independent of \(N\) and \(M_N\).
For smooth functions, \(\mathcal C_\alpha^N\) is the pairing of
\([\partial^\alpha,\mathcal P_R]\Delta F+
\partial^\alpha((\mathcal P_{R^N}-\mathcal P_R)F^N)\)
with \(qh_\alpha\).
\end{lemma}

\begin{proof}
Fix a time at which \(M_N(t)<\infty\) and suppress the time variable.
For smooth functions, \eqref{eq:2.6} gives
\begin{equation}\label{eq:5.3.5}
\begin{aligned}
 [\partial^\alpha,\mathcal P_R]\Delta F
       +\partial^\alpha((\mathcal P_{R^N}-\mathcal P_R)F^N)
 =&-\frac{1}{\gamma}\sum_{0<\nu\le\alpha}\binom\alpha\nu
       (\partial^\nu R)\overline{\partial^{\alpha-\nu}\Delta F}\\
 &-\frac{1}{\gamma}\sum_{\nu\le\alpha}\binom\alpha\nu
       (\partial^\nu\Delta R)\overline{\partial^{\alpha-\nu}F^N}.
\end{aligned}
\end{equation}
For \(\alpha=0\), the pairing with \(qh_0\) equals
\(\frac{1}{\gamma}\mathcal Q_{q,\overline{F^N}}(h_0)\), because
\(\Delta R=\partial_zh_0\).
For \(1\le|\alpha|\le s\), combine the highest coefficient derivatives:
\begin{equation}\label{eq:5.3.6}
 (\partial^\alpha R)\overline{\Delta F}
       +(\partial^\alpha\Delta R)\overline{F^N}
 =(\partial^\alpha R^N)\overline{\Delta F}
       +(\partial^\alpha\Delta R)\overline F.
\end{equation}
Since \(\partial^\alpha\Delta R=\partial_zh_\alpha\), the last term
contributes \(\frac{1}{\gamma}\mathcal Q_{q,\overline F}(h_\alpha)\)
to the weighted pairing in \eqref{eq:5.3.5}.
This proves the identity with \eqref{eq:5.3.3}.
The coefficient differentiated to order \(s\) is thus \(R^N\), not \(R\).

We now estimate \eqref{eq:5.3.3} at the stated regularity.
The flux identity \eqref{eq:2.3} implies
\(\overline F=\partial_x\overline{v\theta}\) and
\(\overline{\Delta F}=\partial_x
\overline{\delta v^N\theta^N+v\delta\theta^N}\).
In particular, \eqref{eq:2.14} and \eqref{eq:2.21} yield
\[
 \|\overline{\Delta F}\|_{H_x^{s-2}}
 \le C\bigl(\|\delta v^N\theta^N\|_{H^{s-1}}
             +\|v\delta\theta^N\|_{H^{s-1}}\bigr)
 \le CX_{s-1}.
\]
Using the same product and averaging estimates at the remaining orders,
together with \eqref{eq:2.18}--\eqref{eq:2.20}, we obtain
\begin{equation}\label{eq:5.3.7}
\begin{alignedat}{2}
&\|q\|_{W^{1,\infty}}+\|R\|_{H^{s-1}}
 +\|R\|_{W^{s-3,\infty}}\le  C,\\
&\|\overline F\|_{W_x^{1,\infty}}
 +\|\overline{F^N}\|_{W_x^{1,\infty}}\le C,\\
&\|\overline{\Delta F}\|_{L_x^\infty}
 +\|\Delta R\|_{H^{s-2}}
 +\|\Delta R\|_{W^{s-4,\infty}}\le CX_{s-1},\\
&\|\overline{\Delta F}\|_{H_x^{s-1}}
 +\|\overline{\Delta F}\|_{W_x^{2,\infty}}
 +\|\Delta R\|_{H^{s-1}}\le CX_s,\\
&\|R^N\|_{H^s}+\|\overline{F^N}\|_{H_x^s}
 +\|\overline{F^N}\|_{W_x^{3,\infty}}\le CM_N.
\end{alignedat}
\end{equation}
These bounds hold also for \(s=4\), with
\(W^{0,\infty}=L^\infty\). The averaged-flux bounds use
\(H_x^{s-1}\hookrightarrow W_x^{2,\infty}\) and
\(H_x^s\hookrightarrow W_x^{3,\infty}\), as stated in
\eqref{eq:2.19}.

Lemma~\ref{lem:vertical-endpoint} bounds the quadratic-form terms by
\(CX_s^2\). The other endpoint satisfies
\[
 \left|\frac{1}{\gamma}\int_\Omega
       q(\partial^\alpha R^N)\overline{\Delta F}\,h_\alpha\right|
 \le CM_NX_{s-1}X_s
 \le CX_s^2+CM_N^2X_{s-1}^2.
\]
For the first sum in \eqref{eq:5.3.3}, the indices
\(s-1-|\nu|\) and \(s-1-|\alpha-\nu|\) are nonnegative and
have sum \(2s-2-|\alpha|\ge s-2>1\).
The balanced product estimate \eqref{eq:2.16}, with averaged functions
identified with their \(z\)-independent extensions, therefore gives
\[
 \|(\partial^\nu R)\overline{\partial^{\alpha-\nu}\Delta F}\|_2
 \le C\|R\|_{H^{s-1}}\|\overline{\Delta F}\|_{H_x^{s-1}}
 \le CX_s.
\]
After pairing with \(qh_\alpha\) and summing, the contribution is at
most \(CX_s^2\).

For the second sum, first take \(|\nu|\le s-2\).
The indices \(s-2-|\nu|\) and \(s-|\alpha-\nu|\) are nonnegative
and again have sum \(2s-2-|\alpha|\ge s-2>1\). Thus
\eqref{eq:2.16} implies
\[
 \|(\partial^\nu\Delta R)\overline{\partial^{\alpha-\nu}F^N}\|_2
 \le C\|\Delta R\|_{H^{s-2}}\|\overline{F^N}\|_{H_x^s}
 \le CM_NX_{s-1}.
\]
If \(|\nu|=s-1\), then \(\nu\ne\alpha\) forces
\(|\alpha|=s\) and \(|\alpha-\nu|=1\).
These products are bounded in \(L^2\) by
\(C\|\Delta R\|_{H^{s-1}}\|\overline{F^N}\|_{W_x^{1,\infty}}
\le CX_s\).
The second sum consequently contributes at most
\(C(M_NX_{s-1}+X_s)X_s\), which is bounded by the right-hand side of
\eqref{eq:5.3.4}.

The preceding product bounds show that every ordinary integral in
\eqref{eq:5.3.3} is absolutely convergent at the stated regularity.
The quadratic terms are defined on \(L^2\) by
Lemma~\ref{lem:vertical-endpoint}.
Summing the estimates over \(|\alpha|\le s\) proves
\eqref{eq:5.3.4}.
\end{proof}

We next combine the projection estimate with the momentum equation
to obtain a one-sided comparison. Write
\(\delta w^N=w^N-w=\delta w^0+\delta\bar w\), where
\(\delta\bar w=-\frac{1}{\gamma}\partial_x
\overline{\delta v^N\theta^N+v\delta\theta^N}\), and define
\[
 E_{\delta,s}^N:=\|\delta v^N\|_{H^s}^2+
       g\beta\sum_{|\alpha|\le s}\int_\Omega q|h_\alpha|^2.
\]
By \eqref{eq:2.13}, \eqref{eq:2.18}, and \eqref{eq:5.3.2},
\(C^{-1}X_s^2\le E_{\delta,s}^N\le CX_s^2\) uniformly on
\([0,\tau]\).

\begin{proposition}\label{prop:one-sided-comparison}
Let \(U,U^N\) be two \(H^s\)-strong solutions satisfying
\eqref{eq:5.3.2} on \([0,\tau]\), and suppose that \(U^N\) is
\(H^{s+1}\)-strong on this interval. Then there exist \(c,C>0\) such that
\begin{equation}\label{eq:5.4.1}
 \sup_{0\le t\le\tau}E_{\delta,s}^N(t)
       +c\int_0^\tau\|\delta v^N\|_{H^{s+1}}^2\,\mathrm dt
 \le CE_{\delta,s}^N(0)
       +C\int_0^\tau M_N^2X_{s-1}^2\,\mathrm dt.
\end{equation}
With \(\tau\), the bounds in \eqref{eq:5.3.2}, and
\(\|v\|_{L^2(0,\tau;H^{s+1})}\) fixed, the constants are
independent of \(N\) and \(M_N\).
\end{proposition}

\begin{proof}
We first derive the estimates for smooth solutions. Set
\(G:=\delta v^N\partial_xv^N+\delta w^N\partial_zv^N\).
Subtracting the momentum equations and the projected temperature
equations \eqref{eq:2.7} yields
\[
\left\{
\begin{aligned}
 &\partial_t\delta v^N+v\partial_x\delta v^N+w\partial_z\delta v^N
       -\mu_x\partial_{xx}\delta v^N-\mu_z\partial_{zz}\delta v^N
       +\partial_x(p^N-p)=-G,\\
 &\partial_t\delta\theta^N+\mathcal P_R\Delta F
       +(\mathcal P_{R^N}-\mathcal P_R)F^N=0,
\end{aligned}
\right.
\]
where
\[
 \Delta F=v\partial_x\delta\theta^N+w^0\partial_z\delta\theta^N
                 +R^N\delta w^0+\delta v^N\partial_x\theta^N.
\]

The diagnostic formulas and \eqref{eq:2.10}, \eqref{eq:2.17},
\eqref{eq:2.21} imply \(\|\delta w^N\|_{H^{s-1}}\le CX_s\).
Together with \eqref{eq:3.2.3}, they also yield
\[
 \|G\|_{H^{s-1}}+
 \|v\partial_x\delta v^N+w\partial_z\delta v^N\|_{H^{s-1}}
 \le CX_s.
\]
As in Proposition~\ref{prop:lower-order-stability}, duality gives,
for every \(\eta>0\),
\begin{equation}\label{eq:5.3.8}
\begin{aligned}
 \left|\sum_{|\alpha|\le s}
       \langle\partial^\alpha G,\partial^\alpha\delta v^N\rangle\right|
 \le{}&C\|G\|_{H^{s-1}}\|\delta v^N\|_{H^{s+1}}\\
 \le{}&\eta\|\delta v^N\|_{H^{s+1}}^2+C_\eta X_s^2.
\end{aligned}
\end{equation}
The brackets denote \(H^{-1}\)--\(H^1\) duality. The same estimate
applies to the reference transport. Thus the complete diagnostic
velocity is controlled without differentiating its homogeneous
component to order \(s\).
Using \eqref{eq:2.12} and the pressure identity \eqref{eq:3.2.11},
applied to the difference without cutoff, we obtain
\begin{equation}\label{eq:5.4.2}
\begin{aligned}
 \frac{1}{2}\frac{\mathrm d}{\mathrm dt}\|\delta v^N\|_{H^s}^2
       +c\|\delta v^N\|_{H^{s+1}}^2
       -g\beta\sum_{|\alpha|\le s}\int_\Omega
                          h_\alpha\partial^\alpha\delta w^0
 \le{}&CE_{\delta,s}^N,
\end{aligned}
\end{equation}
after choosing \(\eta\) sufficiently small.

For the temperature, \(\overline{h_\alpha}=0\), so the weighted
duality \eqref{eq:2.4} removes the principal projection.
The vector field \((v,w^0)\) is divergence-free, and
\eqref{eq:5.2.5} implies
\[
 (\partial_t+v\partial_x+w^0\partial_z)q
 =-q\partial_xv+q^2(\partial_zv)\partial_x\theta
                          +\bar wq^2\partial_zR.
\]
Its right-hand side is bounded in \(L^\infty\) by
\eqref{eq:2.18}, \eqref{eq:3.2.3}, and \eqref{eq:5.3.7}.
The product and commutator estimates \eqref{eq:2.14}--\eqref{eq:2.15},
together with \eqref{eq:2.10} and \eqref{eq:2.18}, yield
\begin{equation}\label{eq:5.4.3}
\begin{aligned}
 \sum_{|\alpha|\le s}\bigl(
 \|[\partial^\alpha,v]\partial_x\delta\theta^N\|_2
 +\|[\partial^\alpha,w^0]\partial_z\delta\theta^N\|_2\bigr)
 \le{}&CX_s+C\|v\|_{H^{s+1}}X_{s-1},\\
 \sum_{|\alpha|\le s}\bigl(
 \|[\partial^\alpha,R^N]\delta w^0\|_2
 +\|\partial^\alpha(\delta v^N\partial_x\theta^N)\|_2\bigr)
 \le{}&CX_s+CM_NX_{s-1}.
\end{aligned}
\end{equation}
For the first estimate, we use
\(\|\nabla\delta\theta^N\|_\infty\le CX_{s-1}\)
to control the factor multiplying the highest-order derivatives
of \(w^0\). The second estimate follows from
\(\|\delta w^0\|_{H^{s-1}}\le CX_s\) and
\(\|\delta w^0\|_\infty+\|\delta v^N\|_\infty\le CX_{s-1}\).
Here the additional regularity is used only to control
\(R^N\) and \(\partial_x\theta^N\) in \(H^s\).

The principal coupling is
\[
R^N\partial^\alpha\delta w^0
=R\partial^\alpha\delta w^0+\Delta R\partial^\alpha\delta w^0.
\]
By \eqref{eq:5.3.7}, \(\|\Delta R\|_\infty\le CX_{s-1}\), and
\begin{equation}\label{eq:5.4.4}
\begin{aligned}
 \left|\sum_{|\alpha|\le s}\int_\Omega
       q\Delta R\partial^\alpha\delta w^0\,h_\alpha\right|
 \le{}&CX_{s-1}X_s\|\delta v^N\|_{H^{s+1}}\\
 \le{}&\eta\|\delta v^N\|_{H^{s+1}}^2+C_{\eta,M}X_s^2.
\end{aligned}
\end{equation}
Here \(X_{s-1}\le CM\) follows from the common bounds, without a
smallness assumption on the difference.
The remaining projected terms have total weighted pairing
\(g\beta\sum_{|\alpha|\le s}\mathcal C_\alpha^N\), which is controlled
by Lemma~\ref{lem:projected-commutator}.
Testing the differentiated temperature equation by \(g\beta qh_\alpha\)
and using \eqref{eq:5.4.3}--\eqref{eq:5.4.4} therefore gives
\begin{equation}\label{eq:5.4.5}
\begin{aligned}
 &\frac{g\beta}{2}\frac{\mathrm d}{\mathrm dt}
       \sum_{|\alpha|\le s}\int_\Omega q|h_\alpha|^2
       +g\beta\sum_{|\alpha|\le s}\int_\Omega
                                \partial^\alpha\delta w^0h_\alpha\\
 \le{}&\eta\|\delta v^N\|_{H^{s+1}}^2
       +C_\eta(1+\|v\|_{H^{s+1}}^2)E_{\delta,s}^N
       +CM_N^2X_{s-1}^2.
\end{aligned}
\end{equation}

We now justify these estimates at the stated regularity.
The strong-solution bounds and \eqref{eq:2.18}, \eqref{eq:2.20}
imply \(q\in L_t^\infty W^{1,\infty}\) and
\(\partial_tq=-q^2\partial_z\partial_t\theta\in C_tL^\infty\).
Apply \(J_L\partial^\alpha\) to the equations and integrate on any
\([t_1,t_2]\subset[0,\tau]\), using
Lemma~\ref{lem:regularized-identities} with weights \(1,q\).
From the expression for \(\Delta F\), one obtains \(\Delta F\in C_tH^{s-1}\) and
\[
\|\Delta F\|_{H^{s-1}}\le CX_s.
\]
Vertical averaging is bounded on \(H^{-1}\) by \eqref{eq:2.8}.
Consequently, \eqref{eq:4.1.1} yields
\[
 [J_L,\mathcal P_R]\partial^\alpha\Delta F
 =-\frac{1}{\gamma}[J_L,R]\overline{\partial^\alpha\Delta F}
 \longrightarrow0\qquad\text{in }L_t^1L^2.
\]
The same result removes
\([J_L,v]\partial_xh_\alpha\) and
\([J_L,w^0]\partial_zh_\alpha\) in \(L_t^1L^2\).
Their principal transport pairings combine with \(\partial_tq\)
through the displayed identity for
\((\partial_t+v\partial_x+w^0\partial_z)q\).

At every finite cutoff, the pressure and principal temperature
coupling satisfy the difference version of \eqref{eq:5.2.7}:
\[
\begin{aligned}
 &\int_\Omega J_L\partial^\alpha\partial_x(p^N-p)
                         J_L\partial^\alpha\delta v^N
       +g\beta\int_\Omega qJ_L(R\partial^\alpha\delta w^0)J_Lh_\alpha\\
 ={}&g\beta\int_\Omega q[J_L,R]\partial^\alpha\delta w^0\,J_Lh_\alpha
       \longrightarrow0\qquad\text{in }L_t^1.
\end{aligned}
\]
The homogeneous pressure has zero pairing with the differentiated
velocity. The convergence follows from \eqref{eq:4.1.1}, since
\(\partial^\alpha\delta w^0\in C_tH^{-1}\).
By \eqref{eq:2.10} and \eqref{eq:5.3.7}, the product
\(\Delta R\partial^\alpha\delta w^0\) belongs to \(L_t^2L^2\),
so its regularized pairing also converges.

For the projection terms, combine the endpoints by
\eqref{eq:5.3.5}--\eqref{eq:5.3.6} in distributions before applying
\(J_L\). For \(1\le|\alpha|\le s\),
Lemma~\ref{lem:vertical-endpoint} and \eqref{eq:5.3.7} imply
\[
 -\frac{1}{\gamma}\int_\Omega
       qJ_L(\overline F\,\partial_zh_\alpha)J_Lh_\alpha
 \longrightarrow\frac{1}{\gamma}\mathcal Q_{q,\overline F}(h_\alpha)
       \qquad\text{in }L_t^1.
\]
For \(\alpha=0\), the same statement holds with
\(\overline F\) replaced by \(\overline{F^N}\).
The remaining temperature remainders belong to \(L_t^2L^2\)
by \eqref{eq:5.3.7}, the product bounds in the proof of
Lemma~\ref{lem:projected-commutator}, and \eqref{eq:5.4.3}.
Their regularized pairings converge by
Lemma~\ref{lem:regularized-identities}, since
\(M_N\in L^2(0,\tau)\) and \(v\in L_t^2H^{s+1}\).

For the momentum forcing, self-adjointness of \(J_L\) and the
duality estimate \eqref{eq:5.3.8} give
\[
\begin{aligned}
 &\int_0^\tau\left|\sum_{|\alpha|\le s}
 \langle\partial^\alpha G,(J_L^2-I)\partial^\alpha\delta v^N\rangle
 \right|\,\mathrm dt\\
 \le{}&C\|G\|_{L_t^2H^{s-1}}
       \|(J_L^2-I)\delta v^N\|_{L_t^2H^{s+1}}
       \longrightarrow0.
\end{aligned}
\]
The same limit holds with \(G\) replaced by
\(v\partial_x\delta v^N+w\partial_z\delta v^N\).
The dissipative terms converge in time, and the regularized energies
converge uniformly, by Lemma~\ref{lem:regularized-identities}.
Thus \eqref{eq:5.4.2} and \eqref{eq:5.4.5} hold in integrated form.

Adding the two estimates cancels the coupling. After absorption,
\[
\begin{aligned}
 &E_{\delta,s}^N(t)+c\int_0^t\|\delta v^N\|_{H^{s+1}}^2\,\mathrm d\rho\\
 \le{} &E_{\delta,s}^N(0)
       +C\int_0^t(1+\|v\|_{H^{s+1}}^2)E_{\delta,s}^N\,\mathrm d\rho
 +C\int_0^tM_N^2X_{s-1}^2\,\mathrm d\rho.
\end{aligned}
\]
The coefficient of the energy has time integral bounded by
\(C(\tau+\|v\|_{L_t^2H^{s+1}}^2)\), independently of \(M_N\),
and \(M_N^2X_{s-1}^2\in L^1(0,\tau)\).
Gronwall's inequality proves \eqref{eq:5.4.1}.
\end{proof}

\subsection{Proof of the main theorem}

We first specialize Proposition~\ref{prop:one-sided-comparison} to
solutions with initial data \(A\) and \(S_NA\). Suppose that these
solutions have uniform \(H^s\)-bounds, velocity dissipation bounds,
and a uniform positive stratification lower bound on a common
interval \([0,\tau]\). Then
\begin{equation}\label{eq:5.4.6}
\begin{aligned}
 \|U(\cdot;S_NA)-U(\cdot;A)\|_{C_t(H^s\times H^s)}
 \le{}&C\bigl(\|(I-S_N)A\|_{H^s\times H^s}\\
 &\qquad+\langle N\rangle
                 \|(I-S_N)A\|_{H^{s-1}\times H^{s-1}}\bigr),
\end{aligned}
\end{equation}
where \(C\) is independent of \(A,N\) when the common bounds and
\(\tau\) are fixed. Indeed, Proposition~\ref{prop:tame-persistence}
ensures that \(U^N=U(\cdot;S_NA)\) is \(H^{s+1}\)-strong on the
whole interval. Taking \(\sigma=s+1\) in \eqref{eq:5.2.8} yields
\begin{equation}\label{eq:5.4.7}
\begin{aligned}
 \int_0^\tau M_N(t)^2\,\mathrm dt
 \le{}&C\tau\bigl(1+\|U^N\|_{C_t(H^{s+1}\times H^{s+1})}^2\bigr)
       +C\|v^N\|_{L_t^2H^{s+2}}^2\\
 \le{}&C\langle N\rangle^2.
\end{aligned}
\end{equation}
For the comparison with \(U=U(\cdot;A)\), the initial difference is
\(\delta U^N(0)=-(I-S_N)A\). Combining \eqref{eq:5.4.7} with
\eqref{eq:4.2.5} at level \(\sigma=s\), we obtain
\[
 \int_0^\tau M_N(t)^2X_{s-1}(t)^2\,\mathrm dt
 \le C\langle N\rangle^2
                  \|(I-S_N)A\|_{H^{s-1}\times H^{s-1}}^2.
\]
Substitution in \eqref{eq:5.4.1}, followed by energy equivalence
and taking square roots, proves \eqref{eq:5.4.6}.

This comparison yields continuous dependence on a uniform local
interval.

\begin{lemma}\label{lem:local-continuity}
Let \(M\ge1\), \(\kappa>0\), and \(A\in\mathcal X_\kappa^s\)
with \(\|A\|_{H^s\times H^s}\le M\). Suppose that
\(A_n\in\mathcal Z^s\) and \(A_n\to A\) in \(H^s\times H^s\).
There exists \(\tau_0>0\) such that \(A_n\in\mathcal X^s\) and
the corresponding maximal solutions are defined on \([0,\tau_0]\)
for all sufficiently large \(n\), with
\begin{equation}\label{eq:5.5.1}
 U(\cdot;A_n)\longrightarrow U(\cdot;A)
       \qquad\text{in }C([0,\tau_0];H^s\times H^s).
\end{equation}
The time \(\tau_0\), the uniform \(H^s\)-bounds and velocity
dissipation bounds, and a common positive stratification lower bound
can be chosen using only \(M,\kappa\) and the fixed parameters.
\end{lemma}
\begin{proof}
By Subsection~\ref{subsec:smoothing-lifespans}, there exist
\(n_0,N_0\ge1\) such that, for \(n\ge n_0\) and \(N\ge N_0\),
\[
 \{A,A_n,S_NA,S_NA_n\}
 \subset
 \left\{B\in\mathcal X_{\frac{\kappa}{2}}^s:
        \|B\|_{H^s\times H^s}\le2M\right\}.
\]
Set \(\tau_0=\frac{1}{2}\tau_*(2M,\frac{\kappa}{2})\).
Proposition~\ref{prop:local-existence} and \eqref{eq:4.1.5}
yield, for every \(B\in\{A,A_n,S_NA,S_NA_n\}\),
\[
 \|U(\cdot;B)\|_{C_t(H^s\times H^s)}
       +\|v(\cdot;B)\|_{L_t^2H^{s+1}}
 \le M_*\!\left(2M,\frac{\kappa}{2}\right)
\]
and
\[
 \inf_{0\le t\le\tau_0}\inf_\Omega R(t;B)
 \ge \frac{\kappa}{8}.
\]

For fixed sufficiently large \(N\), \eqref{eq:5.2.8} with
\(\sigma=s+1\) bounds the two smoothed solutions in \(H^{s+1}\),
uniformly in \(n\). Applying \eqref{eq:4.2.5} at this regularity
level yields
\[
 \|U(\cdot;S_NA_n)-U(\cdot;S_NA)\|_{C_t(H^s\times H^s)}
 \le C_N\|S_N(A_n-A)\|_{H^s\times H^s}\longrightarrow0,
\]
where \(C_N\) is independent of \(n\).
For the two smoothing errors, \eqref{eq:5.4.6} and
\eqref{eq:5.1.2}--\eqref{eq:5.1.3}, at orders \(s\) and \(s-1\),
imply
\[
\begin{aligned}
 &\lim_{N\to\infty}\limsup_{n\to\infty}
 \|U(\cdot;A_n)-U(\cdot;S_NA_n)\|_{C_t(H^s\times H^s)}= 0,\\
 &\lim_{N\to\infty}
 \|U(\cdot;S_NA)-U(\cdot;A)\|_{C_t(H^s\times H^s)}= 0.
\end{aligned}
\]
The triangle inequality proves \eqref{eq:5.5.1}, with \(n\to\infty\)
taken before \(N\to\infty\). The stated uniformity follows from the
choice of \(\tau_0\) and \eqref{eq:4.1.5}.
\end{proof}

We extend this local continuity to every compact subinterval of the
reference lifespan by a finite sequence of restarts.

\begin{proof}[Proof of Theorem~\ref{thm:hadamard}]
Proposition~\ref{prop:maximal-solution} and
Definition~\ref{def:strong-solution}, together with
\eqref{eq:1.2.1} and \eqref{eq:1.2.3} under the normalization
\eqref{eq:1.2.2}, establish assertions \textup{(i)}--\textup{(iii)}.
For \textup{(iv)}, let \(U_{0,n}\in\mathcal Z^s\) converge to
\(U_0\) in \(H^s\times H^s\), and fix \(0<T<T_{\max}(U_0)\).
By \eqref{eq:5.1.4}, \(U_{0,n}\in\mathcal X^s\) for all sufficiently
large \(n\). Write \(U_n(t)=U(t;U_{0,n})\) and
\(U(t)=U(t;U_0)\).
The strong-solution bounds on \([0,T]\) allow us to choose
\(M>1\) and \(\kappa>0\) such that
\[
 \sup_{0\le t\le T}\|U(t)\|_{H^s\times H^s}\le M-1,
 \qquad \inf_{[0,T]\times\Omega}R\ge2\kappa.
\]
Choose \(0<\delta<1\) with \(C_s\delta<\kappa\).
By \eqref{eq:5.1.4}, every \(A\in\mathcal Z^s\) satisfying
\(\|A-U(t)\|_{H^s\times H^s}<\delta\) for some \(t\in[0,T]\)
belongs to \(\mathcal X_\kappa^s\) and has norm at most \(M\).
Lemma~\ref{lem:local-continuity} therefore supplies a common restart
time \(\tau>0\), with uniform bounds, for these fixed \(M,\kappa\).
Choose a finite partition \(0=t_0<t_1<\cdots<t_J=T\) with
\(t_{j+1}-t_j\le\frac{\tau}{2}\).

Assume inductively that \(t_j<T_{\max}(U_{0,n})\) for all
sufficiently large \(n\) and that \(U_n(t_j)\to U(t_j)\) in
\(H^s\times H^s\). These assertions hold at \(j=0\).
For large \(n\), the states \(U_n(t_j)\) lie in the neighborhood
above. Apply the time-translated form of
Lemma~\ref{lem:local-continuity} with initial states
\(U_n(t_j)\) and \(U(t_j)\).
Proposition~\ref{prop:lower-order-stability} identifies the restarted
solutions with the original ones on compact overlaps.
If \(T_{\max}(U_{0,n})<t_j+\tau\), patching the restarted solution
to the earlier one would extend the maximal solution, a contradiction.
Hence, for all sufficiently large \(n\),
\[
 T_{\max}(U_{0,n})\ge t_j+\tau>t_{j+1},
 \qquad
 \|U_n-U\|_{C([t_j,t_{j+1}];H^s\times H^s)}\longrightarrow0.
\]
In particular, \(U_n(t_{j+1})\to U(t_{j+1})\), which completes
the induction. Taking the maximum over the finitely many subintervals
and enlarging the index threshold once, we obtain
\begin{equation}\label{eq:5.5.2}
 U_n\longrightarrow U
       \qquad\text{in }C([0,T];H^s\times H^s).
\end{equation}
For all sufficiently large \(n\), the last restart yields
\[
 T_{\max}(U_{0,n})\ge t_{J-1}+\tau
                  \ge T+\frac{\tau}{2}>T.
\]

Let \(w_n,p_n\) and \(w,p\) denote the corresponding diagnostic
variables, with all pressures normalized by \eqref{eq:1.2.2}.
The diagnostic formulas \eqref{eq:1.2.1}, \eqref{eq:1.2.3} and
the estimates \eqref{eq:2.10}, \eqref{eq:2.17}, \eqref{eq:2.21},
together with \eqref{eq:5.5.2}, yield
\[
 \|w_n-w\|_{C([0,T];H^{s-1})}
       +\|p_n-p\|_{C([0,T];H^s)}
 \le C\|U_n-U\|_{C([0,T];H^s\times H^s)}\longrightarrow0,
\]
where \(C\) is independent of \(n\) by the common \(H^s\)-bounds.
This completes \textup{(iv)}.
For every \(0<T<T_{\max}(U_0)\), the eventually strict lifespan
inequality implies
\(T\le\liminf_{n\to\infty}T_{\max}(U_{0,n})\).
Taking the supremum over these \(T\) gives
\[
 T_{\max}(U_0)\le\liminf_{n\to\infty}T_{\max}(U_{0,n}),
\]
including the case \(T_{\max}(U_0)=\infty\).
The liminf is taken along the admissible tail of the initial-data
sequence. This proves \textup{(v)}.
\end{proof}

%=====================================================================
% End of Section 5
%=====================================================================

\end{document}